\documentclass[12pt]{amsart}
\renewcommand\today{July 21, 2016}
\usepackage{hyperref}
\usepackage[a4paper,margin=1in]{geometry}
\usepackage{amssymb}
\usepackage[all,cmtip]{xy}

\theoremstyle{plain}  
\newtheorem{theorem}{Theorem}[section]
\newtheorem*{theorem*}{Theorem}

\newtheorem*{thma}{Theorem A}
\newtheorem*{thmb}{Theorem B}
\newtheorem*{thm*}{Theorem}

\newtheorem{corollary}[theorem]{Corollary}
\newtheorem{lemma}[theorem]{Lemma}
\newtheorem{proposition}[theorem]{Proposition}

\theoremstyle{definition}
\newtheorem{definition}[theorem]{Definition}

\theoremstyle{remark}

\newtheorem{remark}[theorem]{Remark}

\newtheorem*{claim*}{Claim}

\numberwithin{equation}{section}

\newcommand{\xra}{\xrightarrow}

\newcommand{\abs}[1]{\lvert#1\rvert}
\renewcommand{\leq}{\leqslant}

\renewcommand{\geq}{\geqslant}

\newcommand{\R}{\mathbb{R}}

\newcommand{\C}{\mathbb{C}}

\newcommand{\HH}{\mathbb{H}}

\newcommand{\U}{\mathrm{U}}
\newcommand{\GL}{\mathrm{GL}}

\newcommand{\GCD}{\mathrm{GCD}}

\DeclareMathOperator{\rk}{rk}
\DeclareMathOperator{\im}{im}
\DeclareMathOperator{\coker}{coker}
\DeclareMathOperator{\Hom}{Hom}
\DeclareMathOperator{\End}{End}
\DeclareMathOperator{\Aut}{Aut}
\DeclareMathOperator{\Ext}{Ext}
\DeclareMathOperator{\Id}{Id}

\DeclareMathOperator{\codim}{codim}

\DeclareMathAlphabet{\mathpzc}{OT1}{pzc}{m}{it}

\begin{document}

\title[Birationality of moduli spaces of Higgs bundles]{Birationality of moduli spaces of twisted $\U(p,q)$-Higgs bundles}
\author{Peter B. Gothen}
\author{Azizeh Nozad}
\address{Centro de
  Matem\'atica da Universidade do Porto \\
Faculdade de Ci\^encias da Universidade do Porto \\
Rua do Campo Alegre, s/n \\ 4169-007 Porto \\ Portugal }
\email{pbgothen@fc.up.pt}
\email{azizehnozad@gmail.com}


\thanks{
  Partially supported by CMUP (UID/MAT/00144/2013), the projects
  PTDC/MAT-GEO/0675/2012 and PTDC/MAT-GEO/2823/2014 (first author) and
  grant SFRH/BD/51166/2010 (second author), funded by FCT (Portugal)
  with national and where applicable European structural funds through
  the programme FEDER, under the partnership agreement PT2020.
  The authors acknow\-ledge support from U.S. National Science
  Foundation grants DMS 1107452, 1107263, 1107367 "RNMS: GEometric
  structures And Representation varieties" (the GEAR Network)}

\begin{abstract} 
  A $\U(p,q)$-Higgs bundle on a Riemann surface (twisted by a line
  bundle) consists of a pair of holomorphic vector bundles, together
  with a pair of (twisted) maps between them. Their moduli spaces
  depend on a real parameter $\alpha$. In this paper we study wall
  crossing for the moduli spaces of $\alpha$-polystable twisted
  $\U(p,q)$-Higgs bundles. Our main result is that the moduli spaces
  are birational for a certain range of the parameter and we deduce
  irreducibility results using known results on Higgs bundles. Quiver
  bundles and the Hitchin--Kobayashi correspondence play an essential
  role.
\end{abstract}

\maketitle

\section{Introduction}
Holomorphic vector bundles with extra structure on a Riemann surface
$X$ have been intensively studied over the last decades. Higgs bundles
constitute an important example, not least due to the non-abelian
Hodge Theorem
\cite{corlette:1988,donaldson:1987,Hitchin:1987,simpson:1988,simpson:1992},
which identifies the moduli space of Higgs bundles with the character
variety for representations of the fundamental group.  Another
important example is that of quiver bundles.  A quiver $Q$ is a
directed graph and a $Q$-bundle on $X$ is a collection of vector
bundles, indexed by the vertices of $Q$, and morphisms, indexed by the
arrows of $Q$.  The natural stability condition for quiver bundles
depends on real parameters and hence so do the corresponding moduli
spaces. The stability condition stays the same in chambers but
wall-crossing phenomena arise and can be used in the study of the
moduli spaces. An early spectacular success for this approach is
Thaddeus' proof of the rank two Verlinde formula \cite{thaddeus:1994},
using Bradlow pairs \cite{bradlow:1991}, which are examples of
\emph{triples}. Triples are $Q$-bundles for a quiver with two vertices
and a single arrow connecting them. Moduli spaces of triples have been
studied extensively, using wall-crossing techniques. Without being
exhaustive, we mention
\cite{bradlow-garcia-prada-gothen:2004}, where connectedness
and irreducibility results for triples were studied, and the work of
Mu\~noz \cite{munoz:2008,munoz:2009,munoz:2010} and
Mu\~noz--Ortega--V{\'a}zquez-Gallo
\cite{munoz-ortega-vazques:2007,munoz-ortega-vazques:2009} on finer
topological invariants, such as Hodge numbers. More generally, chains (introduced by \'Alvarez-C\'onsul--Garc{\'\i}a-Prada in \cite{alvarez-garcia-prada:2001})
are $Q$-bundles for a quiver of type $A_n$. Chains have also been
studied using wall crossing techniques; we mention here the work of
Alvar\'ez-Consul--Garc\'\i{}a-Prada--Schmitt \cite{Schmitt:2006},
Garc\'\i{}a-Prada--Heinloth--Schmitt
\cite{garcia-heinloth-schmitt:2014} and Garc\'\i{}a-Prada--Heinloth
\cite{garcia-heinloth:2013}.

A natural question to ask is to what extent wall crossing techniques
can be extended to moduli of $Q$-bundles for more general quivers. Our
aim in this paper is to investigate the situation when $Q$ has
oriented cycles, as opposed to the case of chains.
Since the number of
effective stability parameters is one less than the number of vertices
of the quiver, in order to encounter wall crossing phenomena, we
are led to considering the following quiver as the simplest
non-trivial case:
\begin{equation}
  \label{eq:quiver-Q}
\xymatrix{
\bullet\ar@{<-}@/_1.2pc/[r]&\bullet\ar@{<-}@/_1.2pc/[l].
}
\medskip
\end{equation} 
For such quivers it becomes relevant to consider \emph{twisted}
$Q$-bundles, meaning that to each arrow one associates a fixed line bundle
twisting the corresponding morphism.

Quiver bundles for the quiver (\ref{eq:quiver-Q}) are closely related
to Higgs bundles through the notion of $G$-Higgs bundles. These are
the appropriate objects for extending the non-abelian Hodge Theorem to
representations of the fundamental group in a \emph{real} reductive Lie group
$G$ (see, e.g., \cite{garcia-gothen-mundet:2009b,gothen:2014}). The relevant case here is that of $G=\U(p,q)$. Indeed, a
{$\U(p,q)$-Higgs bundle} is a {twisted} $Q$-bundle for the quiver
(\ref{eq:quiver-Q}), twisted by the canonical bundle $K$ of $X$.
Allowing for twisting by an arbitrary line bundle $L$ on $X$, an
\emph{$L$-twisted $\U(p,q)$-Higgs bundle} is a quadruple
$E=(V,W,\beta,\gamma)$, where $V$ and $W$ are vector bundles of rank
$p$ and $q$, respectively, and the morphisms are $\beta\colon W\to
V\otimes L$ and $\gamma\colon V\to W\otimes L$.  The stability notion
for $Q$-bundles for the quiver \eqref{eq:quiver-Q} depends on a real
parameter $\alpha$ and the value which is relevant for the non-abelian
Hodge Theorem is $\alpha=0$.  

We denote by $\mathcal{M}_\alpha(t)$ the moduli space of
$\alpha$-semistable $L$-twisted $\U(p,q)$-Higgs bundles of type
$t=(p,q,a,b)=(\rk(V),\rk(W),\deg(V),\deg(W))$ and by
$\mathcal{M}^s_\alpha(t) \subset \mathcal{M}_\alpha(t)$ the subspace
of $\alpha$-stable $L$-twisted $\U(p,q)$-Higgs bundles. We show that
the parameter $\alpha$ is constrained to lie in an interval $\alpha_m
\leq \alpha \leq \alpha_M$ (with $\alpha_m=-\infty$ and
$\alpha_M=\infty$ if $p=q$) and the stability condition changes at a
discrete set of critical values $\alpha_c$ for $\alpha$. 

Our main result is the following theorem (see
Theorem~\ref{birational} below).
\begin{thma}
  Fix a type $t=(p,q,a,b)$. Let $\alpha_c$ be a critical value. If
  either one of the following conditions holds:
\begin{itemize}
\item[$(1)$] $a/p-b/q>-\deg(L)$, $q\leq p$ and $0\leq\alpha_c^\pm<\frac{2pq}{pq-q^2+p+q}\big(b/q-a/p-\deg(L)\big)+\deg(L)$,
\item[$(2)$] $a/p-b/q<\deg(L)$, $p\leq q$ and $\frac{2pq}{pq-p^2+p+q}(b/q-a/p+\deg(L))-\deg(L)<\alpha_c^\pm\leq 0$.
\end{itemize}
Then the moduli spaces $\mathcal{M}^s_{\alpha _c^-}(t)$ and
$\mathcal{M}^s_{\alpha _c^+}(t)$ are birationally equivalent. 
\end{thma}

Under suitable co-primality conditions on the topological invariants
$(p,q,a,b)$ we also have results for the full moduli spaces
$\mathcal{M}_\alpha(t)$; we refer to Theorem~\ref{birational}
below for the precise result.

A systematic study of $\U(p,q)$-Higgs bundles was carried out in
\cite{bradlow-garcia-prada-gothen:2003}, based on results for
holomorphic triples from \cite{Bradlow,bradlow-garcia-prada-gothen:2004}. In
particular, it was shown that the moduli space of $\U(p,q)$-Higgs
bundles is irreducible (again under suitable co-primality
conditions). Using these results, we deduce the following corollary to
our main theorem (see Theorem~\ref{irreducible} below).

\begin{thmb}
Let $L=K$ and fix a type $t=(p,q,a,b)$. Suppose that $(p+q,a+b)=1$ and
that $\tau=\frac{2pq}{p+q}(a/p-b/q)$ satisfies $\abs{\tau}\leq\min\{p,q\}(2g-2)$. Suppose that either one of the following conditions holds:
\begin{itemize}
\item[$(1)$] $a/p-b/q>-(2g-2) $, $q\leq p$ and $0\leq\alpha<\frac{2pq}{pq-q^2+p+q}\big(b/q-a/p-(2g-2)\big)+2g-2$,
\item[$(2)$] $a/p-b/q<2g-2$, $p\leq q$ and $\frac{2pq}{pq-p^2+p+q}(b/q-a/p+2g-2)-(2g-2)<\alpha\leq 0$.
\end{itemize}
Then the moduli space $\mathcal{M}_\alpha(t)$ is irreducible.  
\end{thmb}

A related work is the recent preprint by Biquard--Garc\'\i{}a-Prada--R\'ubio
\cite{biquard-garcia-rubio:2015}, which studies $G$-Higgs bundles for
\emph{any} non-compact $G$ of hermitian type. Their focus is different from
ours, in that they adopt a general Lie theoretic approach and study
special properties such as rigidity for maximal $G$-Higgs
bundles, whereas wall crossing phenomena are not studied. On the other
hand it
is similar in spirit in allowing for arbitrary values of the stability
parameter and, indeed, our Proposition~\ref{prop-MW1} for twisted
$\U(p,q)$-Higgs bundles is a special
case of the Milnor--Wood inequality for general $G$ proved by these authors (when $L=K$). A different generalization, namely to \emph{parabolic}
$\U(p,q)$-Higgs bundles, has appeared in the work of
Garc\'\i{}a-Prada--Logares--Mu\~noz \cite{garcia-logares-munoz:2009}.

This paper is organized as follows. In Section~\ref{sec:basic-results}
we give some definitions and basic results on quiver bundles. In
Section~\ref{sec:constraints} we analyze how the $\alpha$-stability
condition constrains the parameter range for fixed type
$t=(p,q,a,b)$, prove the Milnor--Wood type inequality for
$\alpha$-semistable twisted $\U(p,q)$-Higgs bundles mentioned
above, and study vanishing of the second hypercohomology group of
the deformation complex and deduce smoothness results for the moduli
space. These results provide essential input for the
analysis in Section~\ref{sec:crossing} of the loci where the moduli
space changes when crossing a critical value. Finally, in
Section~\ref{bi}, we put our results together and prove our main theorems.

This paper is, in part, based on the second author's Ph.D. thesis \cite{nozad:2016}.

\section{Definitions and basic results}
\label{sec:basic-results}
In this section we recall definitions and relevant facts on quiver bundles, from
\cite{Gothen:2005} and \cite{Prada:2003}, that will be needed in the
sequel. We give the results for general $Q$-bundles. This generality is needed
since more general $Q$-bundles naturally appear in the study of
twisted $\U(p,q)$-Higgs bundles (see Section~\ref{sec:bound-chi}).

\subsection{Quivers} A \emph{quiver} $Q$ is a directed graph specified by a set of vertices $Q_0$, a set of arrows $Q_1$ and head and tail maps $h,t : Q_1\to Q_0$. We shall assume that $Q$ is finite. 
\subsection{Twisted quiver sheaves and bundles}
 Let $X$ be a compact Riemann surface, let $Q$ be a quiver and let $M=\{M_a\}_{a\in Q_1}$ be a collection of finite rank locally free sheaves of $\mathcal{O}_X$-modules. 
\begin{definition}
An \emph{$M$-twisted $Q$-sheaf} on $X$ is a pair $E = (V, \varphi)$, where $V$ is a collection of coherent sheaves $V_i$ on $X$, for each $i \in Q_0$, and $\varphi$ is a collection of morphisms $\varphi_a: V_{ta}\otimes M_a\to V_{ha}$, for each $a\in Q_1$, such that $V_i =0$ for all but finitely many $i \in Q_0$, and $\varphi_a=0$ for all but finitely many $a\in Q_1$.
\par A \emph{holomorphic $M$-twisted $Q$-bundle} is an $M$-twisted $Q$-sheaf $E = (V, \varphi)$ such that the sheaf $V_i$ is a holomorphic vector bundle, for each $i \in Q_0$. 
\par We shall not distinguish vector bundles and locally free finite rank sheaves.
\end{definition}
\par A morphism between twisted $Q$-sheaves $(V, \varphi)$ and $(W, \psi)$ on $X$ is given by a collection of morphisms $f_i:V_i\to W_i$, for each $i\in Q_0$, such that the diagrams 
\[
\xymatrix{
V_{ta}\otimes M_a\ar[r]^{\varphi_a}\ar[d]^{f_{ta}\otimes 1}&V_{ha}\ar[d]^{f_{ha}}\\
W_{ta}\otimes M_a\ar[r]^{\psi_a}&W_{ha}
}
\]
commute for every $a\in Q_1$.
\par In this way $M$-twisted $Q$-sheaves form an abelian category. The
notions of $Q$-subbundles and quotient $Q$-bundles, as well as simple
$Q$-bundles are defined in the obvious way. The subobjects $(0,0)$ and
$E$ itself are called the \emph{trivial subobjects}. The \emph{type}
of a $Q$-bundle $E= (V, \varphi)$ is given by $$t(E)=(\rk(V_i);
\deg(V_i))_{ i\in Q_0},$$ where $\rk(V_i)$ and $\deg(V_i))$ are the
rank and degree of $V_i$, respectively. We sometimes write $\rk(E)=
\rk(\bigoplus V_i)$ and call it the \emph{rank of $E$}. Note that the type is
independent of $\varphi$.
\subsection{Stability}
Fix a tuple $\boldsymbol{\alpha}=(\alpha_i)\in\mathbb{R}^{|Q_0|}$ of
real numbers. For a non-zero $Q$-bundle ${E}=(V,\varphi)$, the
associated \emph{$\boldsymbol{\alpha}$-slope} is defined as
$$\mu_{\boldsymbol{\alpha}}({E})=\frac{\underset{i\in Q_0}{\sum}\big(\alpha_i \rk(V_i)+\deg(V_i)\big)}{\underset{i\in Q_0}{\sum}\rk(V_i)}.$$

\begin{definition}
  A $Q$-bundle ${E}=(V,\varphi)$ is said to be
  \emph{$\boldsymbol{\alpha}$-(semi)stable} if, for all non-trivial
  subobjects ${F}$ of ${E}$,
  $\mu_{\boldsymbol{\alpha}}({F})
    <(\leq)\mu_{\boldsymbol{\alpha}}({E})$. An
  \emph{$\boldsymbol{\alpha}$-polystable} $Q$-bundle is a finite
  direct sum of $\boldsymbol{\alpha}$-stable $Q$-bundles, all of them
  with the same $\boldsymbol{\alpha}$-slope.

  A $Q$-bundle ${E}$ is \emph{strictly
    $\boldsymbol{\alpha}$-semistable} if and only if there is a
  non-trivial subobject ${F}\subset {E}$ such that
  $\mu_{\boldsymbol{\alpha}}({F})=\mu_{\boldsymbol{\alpha}}({E})$.
\end{definition}

\begin{remark}\label{zero}
  If we translate the parameter vector
  $\boldsymbol{\alpha}=(\alpha_i)_{i\in Q_0}$ by a global constant $c
  \in\mathbb{ R}$, obtaining $\boldsymbol{\alpha}'=(\alpha_i')_{i\in
    Q_0}$, with $\alpha_i'=\alpha_i+c$, then
  $\mu_{\boldsymbol{\alpha}'}({E}) =
  \mu_{\boldsymbol{\alpha}}({E})-c$. Hence the stability
  condition does not change under global translations. So we may
  assume that $\alpha_{0}=0$.
\end{remark} 

The following is a well-known fact (see, e.g., \cite[Exercise
$2.5.6.6$]{Schmitt:2008}). Consider a strictly
$\boldsymbol{\alpha}$-semistable $Q$-bundle ${E}=(V,\varphi)$. As it
is not $\boldsymbol{\alpha}$-stable, ${E}$ admits a subobject
${F}\subset {E}$ of the same $\boldsymbol{\alpha}$-slope. If ${F}$ is
a non-zero subobject of ${E}$ of minimal rank and the same
$\boldsymbol{\alpha}$-slope, it follows that ${F}$ is
$\boldsymbol{\alpha}$-stable. Then, by induction, one obtains a flag of
subobjects
\begin{displaymath}
  {F}_0=0\subset {F}_1\subset\cdots\subset {F}_m={E}
\end{displaymath}
where
$\mu_{\boldsymbol{\alpha}}({F}_i/{F}_{i-1})=\mu_{\boldsymbol{\alpha}}({E})$
for $1\leq i\leq m$, and where the subquotients 
$F_i/F_{i-1}$ are $\boldsymbol{\alpha}$-stable $Q$-bundles. This is
the \emph{Jordan-H\"{o}lder filtration} of $\mathcal{E}$, and it is
not unique. However, the \emph{associated graded
  object} 
$$\mathrm{Gr}({E}):=\oplus_{i=1}^m{F}_i/{F}_{i-1}$$
is unique up to isomorphism.
\begin{definition}
  Two semi-stable $Q$-bundles ${E}$ and
  ${E}'$ are said to be \emph{$S$-equivalent} if
  $\mathrm{Gr}(E)\cong \mathrm{Gr}(E').$
\end{definition}

\begin{remark}
  It is a standard fact that each $S$-equivalence class contains a
  unique polystable representative.  Moreover, if a $Q$-bundle ${E}$
  is stable, then the induced Jordan-H\"{o}lder filtration is trivial,
  and so the $S$-equivalence class of $E$ coincides with its
  isomorphism class.
\end{remark}

\subsection{The gauge theory equations}
Throughout this paper, given a smooth bundle $M$ on $X$, $\Omega^k(M)$
(resp.\ $\Omega^{i,j}(M)$) is the space of smooth $M$-valued $k$-forms
(resp.\ $(i,j)$-forms) on $X$, $\omega$ is a fixed K\"{a}hler form on
$X$, and $ \Lambda: \Omega^{i,j}(M)\to \Omega^{i-1,j-1}(M)$ is
contraction with $\omega$. The gauge equations will also depend on a
fixed collection $q$ of Hermitian metrics $q_a$ on $M_a$, for each
$a\in Q_1$, which we fix once and for all. Let
$\mathcal{E}=(V,\varphi)$ be a $M$-twisted $Q$-bundle on $X$. A
Hermitian metric on $\mathcal{E}$ is a collection $H$ of Hermitian
metrics $H_i$ on $V_i$, for each $i\in Q_0$ with $V_i\neq 0$. To
define the gauge equations on $\mathcal{E}$, we note that
$\varphi_a:V_{ta}\otimes M_a\to V_{ha}$ has a smooth adjoint morphism
$\varphi_a^\ast:V_{ha}\to V_{ta}\otimes M_a$ with respect to the
Hermitian metrics $H_{ta}\otimes q_a$ on $V_{ta}\otimes M_a$ and
$H_{ha}$ on $V_{ha}$, for each $a\in Q_1$, so it makes sense to
consider the compositions $\varphi_a\circ\varphi_a^\ast$ and
$\varphi_a^\ast\circ\varphi_a$. The following definitions are found in
\cite{Prada:2003}.  Let $\boldsymbol{\alpha}$ be the stability
parameter.\par Define $\boldsymbol{\tau}$ to be collections of real
numbers $\tau_i$, for which
\begin{equation}
\tau_i=\mu_{\boldsymbol{\alpha}}(\mathcal{E})-\alpha_i, \mbox{ }i\in Q_0.\label{relation}
\end{equation}
Then, by using Remark $\ref{zero}$, $\boldsymbol{\alpha}$ can be recovered from $\boldsymbol{\tau}$ as follows
\begin{align*}
\alpha_i=\tau_{0}-\tau_i, \mbox{ } i\in Q_0.
\end{align*}
\begin{definition}
A Hermitian metric $H$ satisfies the \emph{quiver $\boldsymbol{\tau}$-vortex equations} if
\begin{equation}
\sqrt{-1}\Lambda F(V_i)+\underset{i=ha}{\sum}\varphi_a\varphi_a^{\ast}-\underset{i=ta}{\sum}\varphi_a^{\ast}\varphi_a=\tau_i \mathrm{Id}_{V_i}\label{vortex equations1}
\end{equation}
for each $i\in Q_0$ such that $V_i\neq 0$, where $F(V_i)$ is the
curvature of the Chern connection associated to the metric $H_i$ on
the holomorphic vector bundle $V_i$.
\end{definition}
The following is the \emph{Hitchin-Kobayashi correspondence} between the twisted quiver vortex equations and the stability condition for holomorphic twisted quiver bundles, given in \cite[Theorem 3.1]{Prada:2003}:
\begin{theorem}\label{Hitchin-Kobayashi}
A holomorphic $Q$-bundle ${E}$  is $\boldsymbol{\alpha}$-polystable if and only if it admits a Hermitian metric $H$ satisfying the quiver $\boldsymbol{\tau}$-vortex equations $(\ref{vortex equations1})$, where $\boldsymbol{\alpha}$ and $\boldsymbol{\tau}$ are related by $(\ref{relation})$. 
\end{theorem}
\subsection{Twisted $\U(p,q)$-Higgs bundles}
An important example of twisted $Q$-bundles, which is our main object study in this paper, is that of twisted $\U(p,q)$-Higgs bundles on $X$ given in the following. It is to be noted that twisted $\U(p,q)$-Higgs bundles in our study is twisted with the same line bundle for each arrow. 
\begin{definition}
 Let $L$ be a line bundle on $X$. A \emph{$L$-twisted $\U(p,q)$-Higgs bundle} is a twisted $Q$-bundle for the quiver 
 
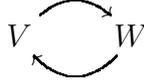
\[ \xymatrix{
V\ar@{<-}@/_1.2pc/[r]&W\ar@{<-}@/_1.2pc/[l].
}
\]
\\
where each arrow is twisted by $L$, and such that $\rk(V)=p$ and
$\rk(W)=q$. Thus a $L$-twisted $\U(p,q)$-Higgs bundle is a quadruple $E=(V, W,\beta, \gamma)$, where $V$ and $W$ are holomorphic vector bundles on $X$ of ranks $p$ and $q$ respectively, and$$\beta:W\longrightarrow V\otimes L,$$ $$ \gamma :V\longrightarrow W\otimes L,$$ are holomorphic maps. The \emph{type} of a twisted $\mathrm{U}(p,q)$-Higgs bundle $E=(V,W,\beta,\gamma)$ is defined by a tuple of integers $t(E):=(p,q,a,b)$ determined by ranks and degrees of $V$ and $W$, respectively.
 \end{definition}
Note that $K$-twisted $\U(p,q)$-Higgs bundles can be seen as a
 special case of $G$-Higgs bundles (\cite{hitchin:1992}, see also \cite{bradlow-garcia-prada-gothen:2003,bradlow-garcia-prada-gothen:hss-higgs,garcia-gothen-mundet:2009b,gothen:2014}), where $G$ is a real form of a
 complex reductive Lie group and $K$ is the canonical bundle of the
 Riemann surface $X$.



\subsection{Gauge equations}
For this $L$-twisted quiver bundle one can consider the general quiver
equations as defined in (\ref{vortex equations1}).
\par Let $\boldsymbol{\tau}=(\tau_1,\tau_2)$ be a pair of real numbers. A Hermitian metric $H$ satisfies the \emph{$L$-twisted quiver $\boldsymbol{\tau}$-vortex equations} on twisted $\U(p,q)$-Higgs bundle $E$ if
\begin{equation}\label{vortex equations}
  \begin{aligned}
    \sqrt{-1}\Lambda F_{H_V}+\beta\beta^\ast-\gamma^\ast\gamma&=\tau_1\Id_{V},\\
    \sqrt{-1}\Lambda F_{H_W}+\gamma\gamma^\ast-\beta^\ast\beta&=\tau_2\Id_{W}.
  \end{aligned}
\end{equation}
where $F_{H_V}$ and $F_{H_W}$ are the curvatures of the Chern connections
associated to the metrics $H_V$ and $H_W$, respectively.
\begin{remark}
\begin{itemize}
\item[(i)] If a holomorphic twisted $\U(p,q)$-bundle $E$ admits a Hermitian metric satisfying the $\boldsymbol{\tau}$-vortex  equations, then taking traces in $(\ref{vortex equations})$, summing for $V$ and $W$, and integrating over $X$, we see that the parameters $\tau_1$ and $\tau_2$ are constrained by $p\tau_1+q\tau_2=\deg(V)+\deg(W)$.
\item[(ii)] If $L=K$ the equations are conformally invariant and so depend only on the Riemann surface structure on $X$. In this case they are \emph{the Hitchin equations} for the $\U(p,q)$-Higgs bundle.
\end{itemize}
\end{remark}
\subsection{Stability}
Let $E=(V,W,\beta,\gamma)$ be a twisted $\U(p,q)$-Higgs bundle, and $\alpha$ be a real number; $\alpha$ is called the $\emph{stability parameter}$. The definitions of the previous section specialize as follows. 
 The \emph{$\alpha$-slope} of $E$ is defined to be 
 \begin{align*}
 \mu_{\alpha}(E)&=\mu(E)+\alpha\frac{p}{p+q},
 \end{align*}
 where $\mu(E):=\mu(V\oplus W)$. A twisted $\U(p,q)$-bundle $E$ is \emph{$\alpha$-semistable} if, for every proper (non-trivial) subobject $F\subset E$, 
$$\mu_{\alpha}(F)\leq\mu_{\alpha}(E).$$
Further, $E$ is \emph{$\alpha$-stable} if this inequality is always strict. A twisted $\U(p,q)$-bundle is called \emph{$\alpha$-polystable} if it is the direct sum of $\alpha$-stable twisted $\U(p,q)$-Higgs bundles of the same $\alpha$-slope.
 \begin{remark}
 The stability can be defined using quotients as for vector bundles. Note that for any subobject $E'=(V',W')$ we obtain an induced quotient bundle $E/E'=(V/V',W/W',\overline{\beta},\overline{\gamma})$ and $E$ is $\alpha$-(semi)stable if $\mu_\alpha(E/E')(\geq)>\mu_\alpha(E)$. 
 \end{remark}
 The following is a special case of the Hitchin-Kobayashi correspondence between the twisted quiver vortex equations and the stability condition for holomorphic twisted quiver bundles, stated in Proposition $\ref{Hitchin-Kobayashi}$.
 \begin{theorem}\label{solution}
 A solution to $(\ref{vortex equations})$ exists if and only if $E$ is $\alpha$-polystable for $\alpha=\tau_2-\tau_1$.
 \end{theorem}
\subsubsection{Critical values}
\par A twisted $\U(p,q)$-Higgs bundle $E$ is strictly $\alpha$-semistable if and only if there is a proper subobject $F=(V',W')$ such that $\mu_\alpha(F)=\mu_\alpha(E)$, i.e.,
$$\mu(V'\oplus W')+\alpha\frac{p'}{p'+q'}=\mu(V\oplus W)+\alpha\frac{p}{p+q}.$$
The case in which the terms containing $\alpha$ drop from the above equality and $E$ is strictly $\alpha$-semistable for all values of $\alpha$, i.e.,
$$\frac{p}{p+q}=\frac{p'}{p'+q'},\mbox{ and}$$
$$\mu(V\oplus W)=\mu(V'\oplus W')$$
is called \emph{$\alpha$-independent strict semistability}.

\begin{definition}
  For a fixed type $(p,q,a,b)$ a value of $\alpha$ is called a
  \emph{critical value} if there exist integers $p', q', a'$ and $b'$
  such that $\frac{p'}{p'+q'}\neq\frac{p}{p+q}$ and
  $\frac{a'+b'}{p'+q'}+\alpha\frac{p'}{p'+q'}=\frac{a+b}{p+q}+\alpha
  \frac{p}{p+q}$, with $0\leq p'\leq p$, $0\leq q'\leq q$
  and $(p',q')\neq (0,0)$. We say that
  $\alpha$ is \emph{generic} if it is not critical.
\end{definition}

\begin{lemma}
\label{lem:no-alpha-independent-semistability}
In the following situations $\alpha$-independent semistability cannot occur:
\begin{itemize}
\item[(i)]\cite[Lemma~2.7]{bradlow-garcia-prada-gothen:2004} There is an integer $m$ such that $\GCD(p+q, d_1+d_2-mp)=1$.
\item[(ii)] If $\GCD(p,q)=1$, for $p\neq q$.
\end{itemize}
\end{lemma}

\begin{proof}
To prove $(ii)$, on the contrary assume that $E=(V,W,\beta,\gamma)$ is a $\alpha$-semistable twisted $\U(p,q)$-Higgs bundle with a proper subobject $E'=(V',W',\beta',\gamma')$ such that
$$
\mu(V'\oplus W')+\alpha\frac{p'}{p'+q'}=
\mu(V\oplus W)+\alpha\frac{p}{p+q}
$$
and
\begin{equation*}
 \dfrac{p'}{p'+q'}=\dfrac{p}{p+q},
\end{equation*}
where $p'$ and $q'$ are the ranks of $V'$ and $W'$ respectively. Since $E'$ is proper, either $p'<p$ or $q'<q$ and then the equality $\dfrac{p'}{p'+q'}=\dfrac{p}{p+q}$ contradicts that $p$ and $q$ are co-prime. 
\end{proof}
Fix a type $t=(p,q,a,b)$. We denote the moduli space of
$\alpha$-polystable twisted $\U(p,q)$-Higgs bundles
$E=(V,W,\beta,\gamma)$ which have the type $t(E)=(p,q,a,b)$, where
$a=\deg(V)$ and $b=\deg(W)$,
by $$\mathcal{M}_\alpha(t)=\mathcal{M}_\alpha(p,q,a,b),$$ and the
moduli space of $\alpha$-stable twisted $\U(p,q)$-Higgs bundles by
$\mathcal{M}_\alpha^s(t)$. A Geometric Invariant Theory construction
of the moduli space 
follows from the general constructions of Schmitt \cite[Theorem
$2.5.6.13$]{Schmitt:2008}, thinking of twisted $\U(p,q)$-Higgs
bundles in terms of quivers.
\subsection{Deformation theory of twisted $\U(p,q)$-Higgs bundles}
\begin{definition}\label{deformation complex of U(p,q)}
Let  $E=(V,W,\beta,\gamma)$ be a $L$-twisted $\U(p,q)$-Higgs bundle and $E'=(V',W',\beta',\gamma')$ a $L$-twisted $\U(p',q')$-Higgs bundle. We introduce the following notation:
\begin{align*}
\mathpzc{Hom}^0&\mathrm{=}\Hom(V',V)\oplus \Hom(W',W), \\
\mathpzc{Hom}^1&\mathrm{=}\Hom(V',W\otimes L)\oplus \Hom(W',V\otimes L).
\end{align*}
With this notation we consider the complex of sheaves 
\begin{equation}
\mathpzc{Hom}^\bullet(E',E):\mbox{ }\mathpzc{Hom}^0\overset{a_0}{\longrightarrow} \mathpzc{Hom}^1\label{deformation complex}
\end{equation}  defined by
\begin{align*}
a_0(f_1,f_2)&=\big(\phi_a(f_1,f_2),\phi_b(f_1,f_2)\big),\mbox{ for } (f_1,f_2)\in \mathpzc{Hom}^0
\end{align*}
where $$\phi_a: \mathpzc{Hom}^0\to\Hom(V',W\otimes L)\hookrightarrow \mathpzc{Hom}^1\mbox{ and }\phi_b: \mathpzc{Hom}^0\to\Hom(W',V\otimes L)\hookrightarrow \mathpzc{Hom}^1$$
are given by 
\begin{align*}
\phi_a(f_1,f_2)&=(f_2\otimes Id_L)\circ \gamma'-\gamma\circ f_1),\\
\phi_b(f_1,f_2)&=(f_1\otimes Id_L)\circ\beta'-\beta\circ f_2).
\end{align*}
The complex $\mathpzc{Hom}^\bullet(E',E)$ is called the
\emph{$\Hom$-complex}. This is a special case of the $\Hom$-complex
for $Q$-bundles defined in \cite{Gothen:2005}, and also for $G$-Higgs
bundles (see, e.g., \cite{Ramanan}). We shall write
$\mathpzc{End}^\bullet(E)$ for $\mathpzc{Hom}^\bullet(E,E)$.
\end{definition}

The following proposition follows from \cite[Theorem 4.1 and Theorem 5.1]{Gothen:2005}.
\begin{proposition}\label{Hom and hyper}
Let  $E$ be a $L$-twisted $\U(p,q)$-Higgs bundle
and $E'$ a $L$-twisted $\U(p',q')$-Higgs
bundle. Then
there are natural isomorphisms 
\begin{align*}
\Hom(E',E)&\cong \mathbb{H}^0(\mathpzc{Hom}^\bullet (E',E))\\
\Ext^1(E',E)&\cong \mathbb{H}^1(\mathpzc{Hom}^\bullet (E',E))
\end{align*}
and a long exact sequence associated to the complex $\mathpzc{Hom}^{\bullet}(E',E)$:
\begin{multline}\label{hyperH}
  0\longrightarrow\mathbb{H}^0(\mathpzc{Hom}^\bullet(E',E))\longrightarrow H^0(\mathpzc{Hom}^0)\longrightarrow H^0(\mathpzc{Hom}^1)\longrightarrow\mathbb{H}^1(\mathpzc{Hom}^\bullet(E',E))\\
  \longrightarrow H^1(\mathpzc{Hom}^0)\longrightarrow H^1(\mathpzc{Hom}^1)\longrightarrow\mathbb{H}^2(\mathpzc{Hom}^\bullet(E',E))\longrightarrow 0. 
\end{multline}
\end{proposition}
When $E=E'$, we have $\End(E)=\Hom(E,E)\cong \mathbb{H}^0(\mathpzc{Hom}^\bullet (E,E))$.
\begin{definition}
We denote by $\chi(E',E)$ the hypercohomology Euler characteristic for the complex $\mathpzc{Hom}^{\bullet}(E',E)$, i.e.
$$\chi(E',E)=\dim\mathbb{H}^0(\mathpzc{Hom}^\bullet(E',E))-\dim\mathbb{H}^1(\mathpzc{Hom}^\bullet(E',E))+\dim\mathbb{H}^2(\mathpzc{Hom}^\bullet(E',E)).$$
\end{definition}
As an immediate consequence of the long exact sequence $(\ref{hyperH})$ and the Riemann-Roch formula we can obtain the following.
\begin{proposition}\label{chi}
For any twisted $\U(p,q)$-Higgs bundle $E$ and twisted $\U(p',q')$-Higgs
bundle $E'$ we have
\begin{equation*}
  \begin{aligned}[t]
    \chi(E',E)&=\chi(\mathpzc{Hom}^0)-\chi(\mathpzc{Hom}^1)\\
    &=(1-g)(\rk(\mathpzc{Hom}^0)-\rk(\mathpzc{Hom}^1))
      +\deg(\mathpzc{Hom}^0)-\deg(\mathpzc{Hom}^1)\\
    &=(1-g)\big(p'p+q'q-p'q-q'p\big)+(q'-p')(\deg(W)-\deg(V))+\\
     &\hspace{1.5cm}(q-p)(\deg(V')-\deg(W'))-(pq'+p'q)\deg(L)
  \end{aligned}
\end{equation*}
\end{proposition}
Recall that the type of a $\U(p,q)$-Higgs bundle
$E=(V,W,\beta,\gamma)$ is $t(E)=(p,q,a,b)$, where $a=\deg(V)$,
$b=\deg(W)$. The previous proposition shows that $\chi(E',E)$ only
depends on the types $t'=t(E')$ and $t=t(E)$ of $E'$ and $E$,
respectively, so we may use the notation
$$\chi(t',t):=\chi(E',E).$$
\begin{remark}
  \label{rem:direct-sum-deform}
  Suppose that $E=E'\oplus E''$. Then it is clear that the
  $\Hom$-complexes satisfy:
\begin{equation*}
\mathpzc{Hom}^\bullet(E,E)=\mathpzc{Hom}^\bullet(E',E')\oplus \mathpzc{Hom}^\bullet(E'',E'')\oplus \mathpzc{Hom}^\bullet(E'', E')\oplus \mathpzc{Hom}^\bullet(E', E''),
\end{equation*}
and so the hypercohomology groups have an analogous direct sum
decomposition. 
\end{remark}

\begin{lemma}\label{extension}
For any extension $0 \rightarrow E' \rightarrow E \rightarrow E'' \rightarrow 0$ of twisted $\U(p,q)$-Higgs bundles,
$$\chi(E,E) = \chi(E', E') + \chi( E'', E'') + \chi( E'', E') + \chi(E', E'').$$
\end{lemma}
\begin{proof}
  Since the Euler characteristic is topological, we may assume that
  $E=E'\oplus E''$. Now the result is immediate in view of
  Remark~\ref{rem:direct-sum-deform}.
\end{proof}

Given the identification of $\mathbb{H}^0(\mathpzc{Hom}^\bullet(E',E))$ with $\Hom(E',E)$, by Proposition $\ref{Hom and hyper}$, the following is the direct analogue of the corresponding result for semistable vector bundles.
\begin{proposition}\label{vanishing}
Let  $E=(V,W,\beta,\gamma)$ be a $L$-twisted $\U(p,q)$-Higgs bundle
and $E'=(V',W',\beta',\gamma')$ a $L$-twisted $\U(p',q')$-Higgs
bundle. If $E$ and $E'$ are both $\alpha$-semistable, then the
following holds:
\begin{itemize}
\item[(1)] If $\mu_\alpha(E)<\mu_\alpha(E')$, then $\mathbb{H}^0(\mathpzc{Hom}^\bullet(E',E))=0$.
\item[(2)] If $\mu_\alpha(E')=\mu_\alpha(E)$, and $E'$ is $\alpha$-stable, then
$$\mathbb{H}^0(\mathpzc{Hom}^\bullet(E',E))\cong\begin{cases} 
0 &\mbox{ if } E\ncong E'\\ 
\mathbb{C}&\mbox{ if } E\cong E'.\\
 \end{cases}$$
 \end{itemize}
 \end{proposition}
 \begin{definition}
A twisted $\U(p,q)$-Higgs bundle $E=(V,W,\varphi=\beta+\gamma)$ is  \emph{infinitesimally simple} if $\End(E) \cong \C$ and it is \emph{simple} if $\Aut(E) \cong \C^\ast$, where $\Aut(E)$ denotes the automorphism group of $E$.
\end{definition}

Since $L$-twisted $\U(p,q)$-Higgs bundles form an abelian category,
any automorphism is also an endomorphism. Hence, if
$(V, W, \beta, \gamma)$ is infinitesimally simple then it is
simple. Thus Proposition~\ref{vanishing} implies the following lemma.

\begin{lemma}\label{stable implies simple}
Let $(V, W, \beta, \gamma)$ be a twisted $\U(p,q)$-Higgs bundle. If $(V,W,\beta,\gamma)$ is $\alpha$-stable 
then it is simple.
\end{lemma}

\begin{proposition}\label{smooth0}
Let $E=(V,W,\beta,\gamma)$ be an $\alpha$-stable twisted $\U(p,q)$-Higgs bundle of type $t=(p,q,a,b)$.
\begin{itemize}
\item[(1)] The space of infinitesimal deformations of $E$ is isomorphic to the first hypercohomology group $\mathbb{H}^1(\mathpzc{End}^\bullet(E))$.\\
\item[(2)]  If $\mathbb{H}^2(\mathpzc{End}^\bullet(E))=0$, then the moduli space $\mathcal{M}^s_\alpha(t)$ is smooth in a
neighborhood of the point defined by $E$ and
\begin{align*}
\dim\mathcal{M}^s_\alpha(t)&=\dim\mathbb{H}^1(\mathpzc{End}^\bullet(E))\\
&=1-\chi(E,E)=(g-1)(q-p)^2+2pq\deg(L)+1.
\end{align*}
\end{itemize}
\end{proposition}
\begin{proof}
Statement $(1)$ follows from \cite[Theorem $2.3$]{Ramanan}. Statement
$(2)$ is analogous to Proposition $2.14$ of \cite{Gothen:2013}. 
\end{proof}

\section{Consequences of stability and properties of the moduli space} 
\label{sec:constraints}

\subsection{Bounds on the topological invariants and Milnor--Wood inequality}

In this section we explore the constraints imposed by stability on the
topological invariants of $\U(p,q)$-Higgs bundles and on the stability
parameter $\alpha$. 

\begin{proposition}\label{bound on beta and gamma}
Let $E=(V,W,\beta,\gamma)$ be an $\alpha$-semistable twisted
$\U(p,q)$-Higgs bundle. Then the following inequalities hold.
\begin{align}
  \label{bound on gamma}
  \frac{2pq}{p+q}\big(\mu(V)-\mu(W)\big) &\leq
  \rk(\gamma)\deg(L)+\alpha\big(\rk(\gamma)-\frac{2pq}{p+q}\big), \\
  \label{bound on beta}
  \frac{2pq}{p+q}\big(\mu(W)-\mu(V)\big)&\leq
  \rk(\beta)\deg(L)+\alpha\big(\frac{2pq}{p+q}-\rk(\beta)\big).
\end{align}
Moreover, if $\deg(L)+\alpha>0$ and equality holds in \eqref{bound on gamma}
then either $E$ is strictly semistable or $p = q$ and
$\gamma$ is an isomorphism $\gamma\colon V \xra{\cong}W\otimes L$. 
Similarly, if $\deg(L)-\alpha>0$ and equality holds in \eqref{bound on
  beta} then either $E$ is strictly semistable or $p = q$  and
$\beta$ is an isomorphism $\beta\colon W\xra{\cong}V\otimes L$.
\end{proposition}

\begin{proof}
An argument similar to that given in \cite[Lemma
3.24]{bradlow-garcia-prada-gothen:2003} shows that 
\begin{equation*}
  2p\big(\mu(V)-\mu_\alpha(E)\big)\leq 
   \rk(\gamma)\deg(L)+\alpha(\rk(\gamma)-2p);
\end{equation*}
Similarly,
\begin{equation*}
  2q\big(\mu(W)-\mu_\alpha(E)\big)
  \leq \rk(\beta)\deg(L)-\rk(\beta)\alpha,
\end{equation*}
Using this, the result follows immediately using the following
identities:
\begin{align*}
  \mu(V)-\mu_\alpha(E)
  &=\frac{q}{p+q}\big(\mu(V)-\mu(W)\big)-\alpha\frac{p}{p+q},\\
  \mu(W)-\mu_\alpha(E)
  &=\frac{p}{p+q}\big(\mu(W)-\mu(V)\big)-\alpha\frac{p}{p+q}.
\end{align*} 
The statement about equality for $\deg(L)-\alpha>0$ also follows as in
loc.\ cit.
\end{proof}

By analogy with the case of $\U(p,q)$-Higgs bundles (cf.\ \cite{bradlow-garcia-prada-gothen:2003}) we make the
following definition.

\begin{definition}
  The \emph{Toledo invariant} of a twisted $\U(p,q)$-Higgs bundle
  $E=(V,W,\beta,\gamma)$ is
  \begin{displaymath}
    \tau(E) = 2\frac{q\deg(V)-p\deg(W)}{p+q} =
    \frac{2pq}{p+q}\big(\mu(V)-\mu(W)\big).
  \end{displaymath}
\end{definition}

The following is the analogue of the Milnor--Wood inequality for
$\U(p,q)$-Higgs bundles
(\cite[Corollary~3.27]{bradlow-garcia-prada-gothen:2003}). When $L=K$,
it is a special case of a general result of
Biquard--Garc\'\i{}a-Prada--Rubio
\cite[Theorem~4.5]{biquard-garcia-rubio:2015}, which is valid for
$G$-Higgs bundles for any semisimple $G$ of Hermitian type.

\begin{proposition}
  \label{prop-MW1}
  Let $E=(V,W,\beta,\gamma)$ be an $\alpha$-semistable twisted
  $\U(p,q)$-Higgs bundle. Then the following inequality holds:
  \begin{displaymath}
    -\rk(\beta)\deg(L)+\alpha\big(\rk(\beta)-\frac{2pq}{p+q}\big)
    \leq \tau(E)\leq
    \rk(\gamma)\deg(L)+\alpha\big(\rk(\gamma)-\frac{2pq}{p+q}\big).
    \end{displaymath}
\end{proposition}

\begin{proof}
In view of the definition of $\tau(E)$, we can
  write \eqref{bound on gamma} and \eqref{bound on beta} as
\begin{align}\label{bound on gamma2}
  \tau(E)&\leq
\rk(\gamma)\deg(L)+\alpha\big(\rk(\gamma)-\frac{2pq}{p+q}\big),\\
  -\tau(E)&\leq
\rk(\beta)\deg(L)+\alpha\big(\frac{2pq}{p+q}-\rk(\beta)\big)
\label{bound on beta2}
\end{align}
  from which the result is immediate.
\end{proof}

When equality holds in the Milnor--Wood inequality, more information
on the maps $\beta$ and $\gamma$ can be obtained from
Proposition~\ref{bound on beta and gamma}. In this respect we have the
following result.

\begin{proposition}
  \label{prop:MW2}
  Let $E=(V,W,\beta,\gamma)$ be an $\alpha$-semistable twisted
  $\U(p,q)$-Higgs bundle. 
  \begin{enumerate}
  \item Assume that $\alpha > -\deg(L)$. Then
    \begin{displaymath}
      \tau(E) \leq \min\{p,q\}\bigl(\deg(L)-\alpha\frac{\abs{p-q}}{p+q}\bigr).
    \end{displaymath}
    and if equality holds then $p\leq q$ and $\gamma$ is an
    isomorphism onto its image.
  \item Assume that $\alpha\leq -\deg(L)$. Then
    \begin{displaymath}
      \tau(E) \leq -\alpha\frac{2pq}{p+q}
    \end{displaymath}
    and if equality holds and $\alpha< -\deg(L)$ then $\gamma=0$.
  \item Assume that $\alpha<\deg(L)$. Then
    \begin{displaymath}
      \tau(E) \geq \min\{p,q\}\bigl(-\alpha\frac{\abs{p-q}}{p+q}-\deg(L)\bigr)
    \end{displaymath}
    and if equality holds then $q\leq p$ and $\beta$ is an
    isomorphism onto its image.
  \item Assume that $\alpha\geq\deg(L)$. Then
    \begin{displaymath}
      \tau(E) \geq -\alpha\frac{2pq}{p+q}
    \end{displaymath}
    and if equality holds and $\alpha > \deg(L)$ then $\beta=0$.
  \end{enumerate}
\end{proposition}

\begin{proof}
  We rewrite \eqref{bound on gamma2} as
  \begin{math}
    \tau(E) \leq \rk(\gamma)(\deg(L)+\alpha) - \alpha\frac{2pq}{p+q}.
  \end{math}
  Then (1) and (2) are immediate from
  Proposition~\ref{bound on beta and gamma}.
  Similarly, (3) and (4) follow rewriting \eqref{bound on beta2} as 
  \begin{math}
    \tau(E) \geq \rk(\beta)(\alpha-\deg(L)) - \alpha\frac{2pq}{p+q}.
  \end{math}
\end{proof}

In the case when $\abs{\alpha}<\deg(L)$ we can write the inequality of
the preceding proposition in a more suggestive manner as follows.

\begin{corollary}
  \label{cor:MW}
  Assume that $\abs{\alpha}<\deg(L)$ and let $E$ be an $\alpha$-semistable
  twisted $\U(p,q)$-Higgs bundle. Then
  \begin{displaymath}
    \abs{\tau(E)} \leq
    \min\{p,q\}\bigl(\deg(L)-\alpha\frac{\abs{p-q}}{p+q}\bigr).
  \end{displaymath}
\end{corollary}

\begin{remark}
  In the cases of Proposition~\ref{prop:MW2} when one of the Higgs
  fields $\beta$ and $\gamma$ is an isomorphism onto its image, it is
  natural to explore rigidity phenomena for twisted $\U(p,q)$-Hitchin
  pairs, along the lines of \cite{bradlow-garcia-prada-gothen:2003}
  (for $\U(p,q)$-Higgs bundles) and Biquard--Garc\'\i{}a-Prada--Rubio
  \cite{biquard-garcia-rubio:2015} (for parameter dependent $G$-Higgs
  bundles when $G$ is Hermitian of tube type). This line of inquiry
  will be pursued elsewhere.
\end{remark}

\subsection{Range for the stability parameter}

In the following we determine a range for the stability parameter whenever $p\neq q$. We denote the minimum and the maximum value for $\alpha$ by $\alpha_m$ and $\alpha_M$, respectively. 

\begin{proposition}\label{bound for alpha}
Assume that $p\neq q$ and let $E$ be a $\alpha$-semistable twisted
$\U(p,q)$-Higgs bundle.  Then $\alpha_m\leq \alpha\leq \alpha_M$,
where
\begin{align*}
  \alpha_m&=
  \begin{cases}
 -\dfrac{2\mathrm{max}\{p,q\}}{|q-p|}\bigl(\mu(V)-\mu(W)\bigr) - \dfrac{p+q}{|q-p|}\deg(L)
  &\text{if}\ \mu(V)-\mu(W)>-\deg(L),\\
  -\bigl(\mu(V)-\mu(W)\bigr) &\text{if}\ \mu(V)-\mu(W)\leq -\deg(L),
\end{cases}\\
\intertext{and}
\alpha_M&=\begin{cases}
-\dfrac{2\mathrm{max}\{p,q\}}{|q-p|}\bigl(\mu(V)-\mu(W)\bigr) +\dfrac{p+q}{|q-p|}\deg(L)&\text{if}\ \mu(V)-\mu(W)<\deg(L),\\
-\bigl(\mu(V)-\mu(W)\bigr) &\text{if}\ \mu(V)-\mu(W)\geq \deg(L).
\end{cases}
\end{align*}
\end{proposition}
\begin{proof}
First we determine $\alpha_M$. Using \eqref{bound on gamma2} we get 
$$\alpha(\dfrac{2pq}{p+q}-\rk(\gamma))\leq\rk(\gamma)\deg(L)-\tau(E)$$
since $p\neq q$ therefore $\dfrac{2pq}{p+q}-\rk(\gamma)> 0$. Hence the above inequality yields 
$$\alpha\leq \dfrac{p+q}{2pq-(p+q)\rk(\gamma)}(\rk(\gamma)\deg(L)-\tau(E)).$$
In order to find an upper bound for $\alpha$ we maximize the right
hand side of this inequality as a function of $\rk(\gamma)$. Thus we
study monotonicity of the function $f(r)=\dfrac{rd-\tau}{c-r}$, where
$c=\dfrac{2pq}{p+q}$, $d=\deg(L)$ and $r\in
[0,\mathrm{min}\{p,q\}]$. We obtain the following:
 \begin{itemize}
 \item[(a)] If $\deg(L)=\mu(V)-\mu(W)$ then $f$ is constant and $$\alpha\leq\mu(W)-\mu(V).$$
 \item[(b)] If $\deg(L)>\mu(V)-\mu(W)$ then $f$ is increasing so 
  \begin{align*}
  \alpha&\leq\dfrac{p+q}{|q-p|}
   \big(\deg(L)-\dfrac{\tau(E)}{\mathrm{min}\{p,q\}}\big)=\dfrac{p+q}{|q-p|}\deg(L)-\dfrac{2\mathrm{max}\{p,q\}}{|q-p|}(\mu(V)-\mu(W))
   \end{align*}
and, if equality holds then $\rk(\gamma)=\mathrm{min}\{p,q\}$.
 \item[(c)] If $\deg(L)<\mu(V)-\mu(W)$ then $f$ is decreasing so
$$\alpha\leq \mu(W)-\mu(V)$$
and, if equality holds then $\gamma=0$.
\end{itemize}
Now we determine the lower bound $\alpha_m$. The  inequality 
\eqref{bound on beta2} yields
$$\alpha\geq \dfrac{\rk(\beta)\deg(L)+\tau(E)}{\rk(\beta)-\dfrac{2pq}{p+q}}.$$
Similarly to the above, by studying the monotonicity of
$g(r)=\dfrac{rd+\tau}{r-c}$, we obtain the following:
 \begin{itemize}
 \item[(a)$'$] If $\mu(V)-\mu(W)=-\deg(L)$ then $g$ is constant and $$\alpha\geq \mu(W)-\mu(V).$$
 \item[(b)$'$] If $\mu(V)-\mu(W)<-\deg(L)$ then $g$ is increasing, so $$\alpha\geq \mu(W)-\mu(V),$$
and, if equality holds then $\beta=0$. 
 \item[(c)$'$] If $\mu(V)-\mu(W)> -\deg(L)$ then $g$ is decreasing, so $$\alpha\geq-\dfrac{p+q}{|q-p|}(\deg(L)+\dfrac{\tau(E)}{\mathrm{min}\{p,q\}})=-\dfrac{p+q}{|q-p|}\deg(L)-\dfrac{2\mathrm{max}\{p,q\}}{|q-p|}(\mu(V)-\mu(W)),$$
and, if equality holds then $\rk(\beta)=\mathrm{min}\{p,q\}$. 
 \end{itemize}
 Note that if $\mu(V)-\mu(W)\geq 0$ then $\mu(V)-\mu(W)\geq -\deg(L)$, and if $\mu(V)-\mu(W)\leq 0$ then $\mu(V)-\mu(W)<\deg(L)$. Hence the result follows. 
\end{proof}
\begin{remark}
  \label{rem:extreme-alpha}
  The preceding proof gives the following additional information when
  $\alpha$ equals one of the extreme values $\alpha_m$ and $\alpha_M$:
  \begin{itemize}
  \item if $\mu(V)-\mu(W)<\deg(L)$ and $\alpha=\alpha_M$ then
    $\rk(\gamma)=\min\{p,q\}$;
  \item if $\mu(V)-\mu(W)>\deg(L)$ and $\alpha=\alpha_M$ then
    $\gamma=0$;
  \item if $\mu(V)-\mu(W)>-\deg(L)$ and $\alpha=\alpha_m$ then
    $\rk(\beta)=\min\{p,q\}$, and
  \item if $\mu(V)-\mu(W)<-\deg(L)$ and $\alpha=\alpha_m$ then
    $\beta=0$.
  \end{itemize}
\end{remark}

The following corollary is relevant because $\alpha=0$ is the value of
stability parameter for which the Non-abelian Hodge Theorem gives the
correspondence between $\U(p,q)$-Higgs bundles and representations of
the fundamental group of $X$.

\begin{corollary}
  With the notation of Proposition~\ref{bound for alpha}, the inequality $\alpha_M\geq 0$ holds if and only if $\tau(E)\leq
  \min\{p,q\}\deg(L)$
  and the inequality $\alpha_m\leq 0$ holds if and only if $\tau(E)\geq
  -\min\{p,q\}\deg(L)$.
  Thus $0\in[\alpha_m,\alpha_M]$ if and only if $\abs{\tau(E)}\leq
  \min\{p,q\}\deg(L)$.
\end{corollary}

\begin{proof}
  Immediate from Proposition~\ref{bound for alpha}.
\end{proof}

\begin{remark}
  Note that the condition $\abs{\tau(E)}\leq \min\{p,q\}\deg(L)$ is
  stronger than the condition $\abs{\mu(V)-\mu(W)}\leq \deg(L)$.
\end{remark}

\subsection{Parameters forcing special properties of the Higgs fields}
\label{se:critical-values}

In this section we use a variation on the preceding arguments to find
a parameter range where $\beta$ and $\gamma$ have special
properties. Assume that the twisted $\mathrm{U}(p,q)$-Higgs bundle $E=(V,W,\beta,\gamma)$ has type $(p,q,a,b)$.

For the following proposition it is convenient to introduce the
following notation. For $0\leq i < q \leq p$, let
\begin{displaymath}
  \alpha_i=\dfrac{2pq}{q(p-q)+(i+1)(p+q)}
  \big(\mu(W)-\mu(V)-\deg(L)\big)+\deg(L),  
\end{displaymath}
and for $0\leq j< p\leq q$, let 
\begin{displaymath}
  \alpha_j'=\dfrac{2pq}{p(q-p)+(j+1)(p+q)}
  \big(\mu(W)-\mu(V)+\deg(L)\big)-\deg(L).
\end{displaymath}




\begin{proposition}\label{surjective}
Let $E=(V,W,\beta,\gamma)$ be an $\alpha$-semistable twisted $\U(p,q)$-Higgs bundle. Then we have the following:
\begin{itemize}

\item[$(i)$] Assume that $p\geq q$ and $\mu(V)-\mu(W)>-\deg(L)$. If
  $\alpha<\alpha_{i-1}$ then $\rk(\ker(\beta))<i$. In particular
  $\beta$ is injective whenever
  \begin{displaymath}
    \alpha<\alpha_0 = \frac{2pq}{pq-q^2+p+q}
    \big(\mu(W) - \mu(V) - \deg(L)\big) + \deg(L).
  \end{displaymath}

\item[$(ii)$] Assume that $p\geq q$ and $\mu(V)-\mu(W)<-\deg(L)$. If
  $\alpha<\alpha_{i-1}$ then $\rk(\ker(\beta))>i$. In particular
  $\beta$ is zero whenever
  \begin{displaymath}
    \alpha<\alpha_{q-2} = \frac{2pq}{2pq-p-q}
    \big(\mu(W) - \mu(V) - \deg(L)\big) + \deg(L).
  \end{displaymath}

\item[$(iii)$] Assume that $p\leq q$ and $\mu(V)-\mu(W)<\deg(L)$. If
  $\alpha>\alpha'_j$ then $\rk(\ker(\gamma))<j$. In particular
  $\gamma$ is injective whenever
  \begin{displaymath}
    \alpha>\alpha_0' = \frac{2pq}{pq-p^2+p+q}
    \big(\mu(W) - \mu(V) + \deg(L)\big) - \deg(L).
  \end{displaymath}

\item[$(iv)$] Assume that $p\leq q$ and $\mu(V)-\mu(W)>\deg(L)$. If
  $\alpha>\alpha'_j$ then $\rk(\ker(\gamma))>j$. In particular
  $\gamma$ is zero whenever
  \begin{displaymath}
    \alpha>\alpha_{p-2}' = \frac{2pq}{2pq-p-q}
    \big(\mu(W) - \mu(V) + \deg(L)\big) - \deg(L).
  \end{displaymath}
\end{itemize}
\end{proposition}

\begin{proof}
We shall only prove parts $(i)$ and $(ii)$. One can deduce the other parts in a similar way. Suppose that $\rk(\ker(\beta))=n>0$. The inequality $(\ref{bound on beta})$ yields 
\begin{displaymath}
  \alpha \geq \frac{2pq}{n(p+q)+q(p-q)}
  \big(\mu(W)-\mu(V)-\deg(L)\big)+\deg(L) = \alpha_{n-1}.
\end{displaymath}

Now suppose $\mu(W)-\mu(V)-\deg(L)<0$, then $\alpha_i$ increases
with $i$ and so, if $n\geq i$ then $\alpha \geq \alpha_{i-1}$. Hence, if
$\alpha < i-1$ then $n<i$.  In particular, if $\alpha<\alpha_0$ then $\beta$ is injective, which gives part $(i)$.

On the other hand, if $\mu(W)-\mu(V)-\deg(L)>0$, then $\alpha_i$ decreases
with $i$ and so, if $n\leq i$ then $\alpha \geq \alpha_{i-1}$. Hence, if
$\alpha < \alpha_{i-1}$ then $n>i$.  In particular, if $\alpha < \alpha_{q-2}$ then $\beta$ is zero, proving part $(ii)$.
\end{proof}

\begin{remark}
  \label{rem:alpha-0-positive}
Note that the signs of $\alpha_0$ and $\alpha'_0$ given in the
preceding proposition are related to the Toledo invariant as follows:
\begin{itemize}
\item $\alpha_0>0$ if and only if $\tau(E)<-(q-1)\deg(L)$.
\item $\alpha_0'<0$ if and only if $\tau(E)>(p-1)\deg(L)$.
\end{itemize}
\end{remark}





\begin{remark}
  \label{rem:upq-duality}
  Associated to $E=(V,W,\beta,\gamma)$ there is a dual $L$-twisted
  $\U(p,q)$-Higgs bundle $E^\ast=(V^*,W^*,\gamma^*,\beta^*)$. Clearly
  there is a one-to-one correspondence between subobjects of $E$ and
  quotients of $E^*$, and
  $\mu_{-\alpha}(E)=-\mu_\alpha(E^*)$. Therefore $\alpha$-stability of
  $E^*$ is equivalent to $-\alpha$-stability of $E$.
\end{remark}

\begin{corollary}\label{Cor:Surjective}
Let $E=(V,W,\beta,\gamma)$ be an $\alpha$-semistable twisted $\U(p,q)$-Higgs bundle. Then we have the following:
\begin{itemize}
\item[(i)] If $p\geq q$ and $\mu(W)-\mu(V)>-\deg(L)$ then $\gamma$ is surjective whenever $$\alpha>\alpha_t:=\frac{2pq}{pq-q^2+p+q}(\mu(W)-\mu(V)+\deg(L))-\deg(L).$$
\item[(ii)] If $p\leq q$ and $\mu(W)-\mu(V)<\deg(L)$ then $\beta$ is surjective whenever $$\alpha<\alpha_t':=\frac{2pq}{pq-p^2+p+q}(\mu(W)-\mu(V)-\deg(L))+\deg(L).$$
\end{itemize}
\end{corollary}

\begin{proof}
  Using Proposition~\ref{surjective} we can find a range for the
  stability parameter of $E^*$ where $\beta^*$ and $\gamma^*$ are
  injective. Hence the result follows by using
  Remark~\ref{rem:upq-duality} to relate the stability parameters of
  $E$ and $E^*$.
\end{proof}

The following results shows that the bounds in
Proposition~\ref{surjective} are meaningful in view of the bounds for
$\alpha$ of Proposition~\ref{bound for alpha}.

\begin{proposition}
Let $\alpha_0$ and $\alpha_0'$ be given in
Proposition~\ref{surjective}. Then the following holds.
\begin{itemize}
\item[$(i)$] Assume that $p>q$. If $\mu(V)-\mu(W) > -\deg(L)$ then
  $\alpha_0>\alpha_m$, and if $\mu(V)-\mu(W) < -\deg(L)$ then $\alpha_{q-2}>\alpha_m$.
\item[$(ii)$] Assume that $p<q$. If $\mu(V)-\mu(W) < \deg(L)$ then
  $\alpha_0'<\alpha_M$, and if $\mu(V)-\mu(W) > \deg(L)$ then
  $\alpha'_{p-2}<\alpha_M$. 
\end{itemize}
\end{proposition}
\begin{proof}
For $(i)$, using $\mu(V)-\mu(W)> -\deg(L)$ we get 
\begin{align*}
  \alpha_0-\alpha_m&=\bigl(\mu(V)-\mu(W)\bigr)
    \Bigl(\dfrac{-2pq}{q(p-q)+p+q}+\dfrac{2p}{p-q}\Bigr)\\
  &\phantom{==}+\deg(L)
    \Bigl(\dfrac{-2pq}{q(p-q)+p+q}+1+\dfrac{p+q}{p-q}\Bigr)\\
  &>\deg(L)\Bigl(-\dfrac{2p}{p-q}+1+\dfrac{p+q}{p-q}\Bigr)=0,
\end{align*}
where we have used that $p>q$ makes the term which multiplies
$\mu(V)-\mu(W)$ positive. Thus $\alpha_0>\alpha_m$. Moreover, when $\mu(V)-\mu(W) < -\deg(L)$
and $p>q$, we have $\alpha_m=\alpha_{q-1}<\alpha_{q-2}$ (cf.\ the
proof of Proposition~\ref{surjective}). This finishes
the proof of $(i)$.

For $(ii)$, using $\mu(V)-\mu(W)<\deg(L)$ we obtain the following 
\begin{align*}
  \alpha_M-\alpha_0'&=\bigl(\mu(V)-\mu(W)\bigr)
    \Bigl(\dfrac{-2q}{q-p}+\dfrac{2pq}{p(q-p)+p+q}\Bigr)\\
  &\phantom{==}+\deg(L)
    \Bigl(\dfrac{p+q}{q-p}-\dfrac{2pq}{p(q-p)+p+q}+1\Bigr)\\
  &>\deg(L)\Bigl(-\dfrac{2q}{q-p}+1+\dfrac{p+q}{q-p}\Bigr)=0,
\end{align*}
where we have used that $p<q$ makes the term which multiplies
$\mu(V)-\mu(W)$ negative.
Hence $\alpha_0'<\alpha_M$.  Moreover, when $\mu(V)-\mu(W) > \deg(L)$
and $p<q$, we have $\alpha_M=\alpha'_{p-1}>\alpha'_{p-2}$ (again, cf.\ the
proof of Proposition~\ref{surjective}). This finishes
the proof of $(ii)$.
\end{proof}

\subsection{The comparison between $\U(p, q)$-Higgs bundles and $\GL(p + q, \C)$-Higgs bundles}
Any $\U(p,q)$-Higgs bundle gives rise to a $\GL(p+q,\C)$-Higgs
bundle. In this section we compare the respective stability
conditions. We shall not need these results in the remainder of the
paper but for completeness we have chosen to include them, since the
question is a natural one to consider.

We recall the following about $\GL(n, \C)$-Higgs bundles.
A $\GL(n, \C)$-Higgs bundle on $X$ is a pair $(E, \phi)$, where $E$ is
a rank $n$ holomorphic vector bundle over $X$ and
$\phi\in H^0(\End(E)\otimes K)$ is a holomorphic endomorphism of $E$
twisted by the canonical bundle $K$ of $X$.  More generally, replacing
$K$ by an arbitrary line bundle on $X$, we obtain the notion of a
$L$-twisted $\GL(n,\C)$-Higgs bundle on $X$.
The $\GL(n,\C)$-Higgs bundle $(E,\phi)$ is stable if the slope
stability condition
\begin{displaymath}
\mu(E') <\mu(E)
\end{displaymath}
holds for all non-zero proper $\phi$-invariant subbundles $E'$ of
$E$. Semistability is defined by replacing the strict inequality with
a weak inequality. A twisted Higgs bundle is called polystable if it is the
direct sum of stable twisted Higgs bundles with the same slope.

\begin{remark}
  Nitsure \cite{nitsure:1991} was the first to study twisted Higgs
  bundles in a systematic way. For some of his results he needs to
  make the assumption $\deg(L)\geq 2g-2$ (similarly, for example, to
  our Proposition~\ref{vanishing main} below). However, the 
  comparison of stability conditions which we carry out here is valid
  for any $L$.
\end{remark}

For any twisted $\U(p,q)$-Higgs bundle $E=(V,W,\beta,\gamma)$ we can
associate a twisted $\GL(p+q,\C)$-Higgs bundle defined by taking
$\widetilde{E}=V\oplus W$ and
$\phi=\left( \begin{array}{cc}0 &\beta \\ \gamma& 0\end{array}
\right)$.

The following result is reminiscent of Theorem~3.26 of
\cite{Gothen:2012}, which is a result for $\mathrm{Sp}(2n, \R)$-Higgs
bundles. The corresponding result for $0$-semistable $\U(p,q)$-Higgs
bundles can be found in the appendix to the first preprint version of
\cite{Gothen:2013} and the proof given there easily adapts to the
present situation. We include it here for the convenience of the
reader.

Recall from Proposition~\ref{surjective} that for $p=q$, 
\begin{align}
  \label{eq:3}
  \alpha_0&=p\big(\mu(W)-\mu(V)-\deg(L)\big)+\deg(L),\\
  \label{eq:4}
  \alpha_0'&=p\big(\mu(W) - \mu(V) + \deg(L)\big) - \deg(L).
\end{align} 




\begin{proposition}
  \label{prop:upp-alpha-gl-stability}
  Let $E=(V,W,\beta,\gamma)$ be an $\alpha$-semistable twisted
  $\U(p,q)$-Higgs bundle such that $p=q$ and let $\alpha_0$ and
  let $\alpha_0'$ be given by \eqref{eq:3} and \eqref{eq:4},
  respectively. Suppose that one of the following conditions holds:
\begin{enumerate}
  \item $\mu(V)-\mu(W)>-\deg(L)$ and $0\leq\alpha<\alpha_0$.
  \item $\mu(V)-\mu(W)<\deg(L)$ and $\alpha_0'<\alpha\leq 0$. 
\end{enumerate}
Then the associated $\GL(2p,\C)$-Higgs bundle $\widetilde{E}$ is
semistable. Moreover $\alpha$-stability of $E$ implies stability of
$\widetilde{E}$ unless there is an isomorphism $f:V\to W$ such that
$\beta f=f^{-1}\gamma$. In this case $(\widetilde{E},\phi)$ is
polystable and decomposes as
$$
(\widetilde{E},\phi)=(\widetilde{E}_1,\phi_1)\oplus
(\widetilde{E}_2,\phi_2)
$$
where each summand is a stable $\GL(p,\C)$-Higgs bundle isomorphic to $(V,\beta f)$.
\end{proposition}

\begin{proof}
  Let $\widetilde{E}'$ be an invariant subbundle of
  $\widetilde{E}$. By projecting onto $V$ and $W$ and taking the
  kernels and images, we get the following short exact sequences:
\begin{align}\label{eq:short-exact-sequance}
&0\to W''\to \widetilde{E}'\to V'\to 0,\nonumber\\
&0\to V''\to \widetilde{E}'\to W'\to 0.
\end{align}
We can then deduce that
\begin{align}\label{eq:s1}
  \deg W'' +\deg V' &=\deg\widetilde{E}' = \deg V'' +\deg W'\nonumber\\
  q''+p' &=\rk\widetilde{E}' =p''+q'
\end{align}
where $q''$, $q'$, $p''$ and $p'$ denote the rank of $W''$, $W'$, $V''$ and $V'$, respectively. Note that $(V',W')$ and $(V'',W'')$ define subobjects of $E$. The $\alpha$-semistability conditions applied to these subobjects imply
\begin{align}
\deg V' +\deg W'&\leq \mu(E)(p'+q')+\frac{q'-p'}{2}\alpha\label{eq:slope-condition-E'}\\
\deg V'' +\deg W''&\leq \mu(E)(p''+q'')+\frac{q''-p''}{2}\alpha\label{eq:slope-condition-E''}
\end{align}
Adding these two inequalities and using \eqref{eq:s1}, we get
\begin{equation}
  \label{eq:slope-tilde-E-prime}
  \mu(\widetilde{E}')\leq
  \mu(\widetilde{E})+\frac{q'-p'+q''-p''}{2(p'+p''+q'+q'')}\alpha
  = \mu(\widetilde{E})+\frac{q'-p'}{p'+p''+q'+q''}\alpha
\end{equation}
From Proposition~\ref{surjective} we obtain the injectivity of $\beta$
and $\gamma$ by using the hypotheses $(1)$ and $(2)$,
respectively. Injectivity of $\beta$ and $\gamma$ yield $q'\leq p'$
and $q'\geq p'$, respectively. Hence, in either case $(q'-p')\alpha$
is negative. Therefore \eqref{eq:slope-tilde-E-prime} proves that
$\widetilde{E}$ is semistable. \par
Suppose now that $E$ is $\alpha$-stable. Then, by the above argument,
$\widetilde{E}$ is semistable and it is stable if
\eqref{eq:slope-tilde-E-prime} is strict for all non-trivial
subbundles $\widetilde{E}'\subset \widetilde{E}$. The equality holds
in \eqref{eq:slope-tilde-E-prime} if it holds in both
\eqref{eq:slope-condition-E'} and
\eqref{eq:slope-condition-E''}. Since $E$ is $\alpha$-stable the only
way in which a non-trivial subbundle
$\widetilde{E}'\subset \widetilde{E}$ can yield equality in
\eqref{eq:slope-tilde-E-prime} is that
$$V'\oplus W'=V\oplus W\mbox{ and }V''\oplus W''.$$ 
In this case from \eqref{eq:short-exact-sequance} we obtain
isomorphisms $E'\to V$ and $E'\to W$. Therefore, combining these, we
get an isomorphism $f:V\to W$ such that $\beta f=f^{-1}\gamma$. Hence,
if there is no such isomorphism between $V$ and $W$ then $(\widetilde{E},\phi)$ is $\alpha$-stable.

Now suppose that there exists such an isomorphism $f:V\to W$, define
\begin{align*}
(\widetilde{E}_1,\phi_1)&=(\{(v,f(v))\in\widetilde{E}|v\in V\}, \phi|_{\widetilde{E}_1}),\\
(\widetilde{E}_2,\phi_2)&=(\{(v,-f(v))\in\widetilde{E}|v\in V\}, \phi|_{\widetilde{E}_2}).
\end{align*}
The fact that $\beta f=f^{-1}\gamma$ implies that $(E_i,\phi_i)$, $i=1, 2$, define $\GL(n,\C)$-Higgs bundles isomorphic to $(V,\beta f)$. We have $$(\widetilde{E},\phi)=(\widetilde{E}_1,\phi_1)\oplus(\widetilde{E}_2,\phi_2),$$
with $$\mu(\widetilde{E}_1)=\mu(\widetilde{E})=\mu(\widetilde{E}_2).$$
To show that each summand is a stable $\GL(n,\C)$-Higgs bundle, note that any non-trivial subbundle $\widetilde{E}'$ of $\widetilde{E}_i$ is a subbundle of $\widetilde{E}$ and hence $\mu(\widetilde{E}')<\mu(\widetilde{E})=\mu(\widetilde{E}_i)$. 
\end{proof}

\begin{remark}
\label{rem:comparision}
We can also conclude from the proof of the above proposition that a
twisted $\U(p,q)$-Higgs bundle is $\alpha$-semistable for $\alpha=0$
if and only if the associated $\GL(p+q,\C)$-Higgs bundle is
semistable. Equivalence also holds for stability, unless there is
an isomorphism $f:V\to W$ such that $\beta f=f^{-1}\gamma$. 
\end{remark}

\subsection{Vanishing of hypercohomology in degree two}
In order to study smoothness of the moduli space we investigate
vanishing of the second hypercohomology group of the deformation
complex (cf.\ Proposition~\ref{smooth0}). This vanishing will also
play an important role in the analysis of the flip loci in
Section~\ref{sec:crossing}. We note that vanishing is not guaranteed by
$\alpha$-stability for $\alpha\neq 0$, in contrast to the case of
triples (and chains), where vanishing is guaranteed for $\alpha>0$.

By using the obvious symmetry of the quiver interchanging the vertices
we can associate to a $\U(p,q)$-Higgs bundle a $\U(q,p)$-Higgs
bundle. The following proposition is immediate.

\begin{proposition}
  \label{prop:sigma-alpha-stability}
  Let $E=(V,W,\beta,\gamma)$ be a $\U(p,q)$-Higgs bundle and let
  $\sigma(E) = (W,V,\gamma,\beta)$ be the associated $\U(q,p)$-Higgs
  bundle. Then $E$ is $\alpha$-stable if and only if $\sigma(E)$
  is $-\alpha$-stable, and similarly for poly- and
  semi-stability. \qed
\end{proposition}

The next result uses this construction and Serre duality to identify the second
hypercohomology of the Hom-complex with the dual of a zeroth
hypercohomology group.

\begin{lemma}
  \label{lem:H2-is-H0}
  Let  $E=(V,W,\beta,\gamma)$ be a $L$-twisted $\U(p,q)$-Higgs bundle
  and $E'=(V',W',\beta',\gamma')$ a $L$-twisted $\U(p',q')$-Higgs
  bundle. Let
  $E''=\sigma(E')\otimes L^{-1}K=(W'\otimes L^{-1}K,V'\otimes
  L^{-1}K,\gamma\otimes 1,\beta\otimes 1)$. Then
  \begin{displaymath}
    \HH^2(\mathpzc{Hom}^\bullet(E',E))
      \cong \HH^0(\mathpzc{Hom}^\bullet(E,E''))^*.
  \end{displaymath}
\end{lemma}

\begin{proof}
By Serre duality for hypercohomology  
$$
\mathbb{H}^2(\mathpzc{Hom}^\bullet(E',E))
\cong 
\mathbb{H}^0(\mathpzc{Hom}^{\bullet^\vee}(E',E)\otimes K)^*
$$
where the dual complex twisted by $K$ is
\begin{multline*}
\mathpzc{Hom}^{\bullet^\vee}(E',E)\otimes K:\big(\Hom(V,W'\otimes
L^{-1})\oplus\Hom(W,V'\otimes L^{-1})\big)\otimes K \\ \to\big(\Hom(V,V')\oplus\Hom(W,W')\big)\otimes K.
\end{multline*}
One easily checks that the differentials correspond, so that
\begin{displaymath}
  \mathpzc{Hom}^{\bullet^\vee}(E',E)\otimes K \cong
  \mathpzc{Hom}^\bullet(E,E'').
\end{displaymath}
This completes the proof.
\end{proof}

\begin{lemma}
  \label{lem:lambda-f}
  Let  $E=(V,W,\beta,\gamma)$ be a $L$-twisted $\U(p,q)$-Higgs bundle
  and $E'=(V',W',\beta',\gamma')$ a $L$-twisted $\U(p',q')$-Higgs
  bundle. As above
  let
  $E''=\sigma(E')\otimes L^{-1}K = (W'\otimes L^{-1}K,V'\otimes
  L^{-1}K,\gamma'\otimes 1,\beta'\otimes 1)$.
  Let $f\in  \HH^0(\mathpzc{Hom}^\bullet(E,E''))$ viewed as
  morphism $f\colon E\to E''$ and write
  $\lambda(f)=\frac{\rk(f(V))}{\rk(f(V))+\rk(f(W))}$. Then, if
  $f\neq 0$, the inequality
  \begin{equation}
    \label{eq:h2-vanishing-ineq}
    \alpha(2\lambda(f)-1)+2g-2-\deg(L) \geq 0
  \end{equation}
  holds. Moreover, if $E$ and $E''$ are $\alpha$-stable, then strict
  inequality holds unless $f\colon E\xra{\cong}E''$ is an isomorphism.
\end{lemma}

\begin{proof}
  Write $N=\ker(f) \subset
  E$ and $I=\im(f) \subset E''$. Then $\alpha$-semistability of $E$
  implies that $\mu_\alpha(N)\leq\mu_\alpha(E)$, which is equivalent
  to
  \begin{equation}
    \label{eq:1}
    \mu_\alpha(I) \geq \mu_\alpha(E);
  \end{equation}
  note that this also holds if $N=0$, since then $I\cong E$.
  Moreover, by Proposition~\ref{prop:sigma-alpha-stability}, $E''$ is
  $-\alpha$-semistable and so
  \begin{math}
    \mu_{-\alpha}(I)\leq \mu_{-\alpha}(E'').
  \end{math}
  This, using that $\mu_{-\alpha}(I) = \mu_\alpha(I)
  -2\alpha\lambda(f)$ and $\mu_{-\alpha}(E'') =
  \mu_\alpha(E)-\alpha+(2g-2-\deg(L))$, is equivalent to
  \begin{equation}
    \label{eq:2}
    \mu_\alpha(I) \leq \mu_\alpha(E) + 2\alpha\lambda(f) - \alpha +
    2g-2 - \deg(L).
  \end{equation}
  Combining \eqref{eq:1} and \eqref{eq:2} gives the result. The
  statement about strict inequality is easy.
\end{proof}



The following is our first main result on vanishing of $\HH^2$. It
should be compared with
\cite[Proposition~3.6]{bradlow-garcia-prada-gothen:2004}. The reason
why extra conditions are required for the vanishing is essentially
that the ``total Higgs field'' $\beta+\gamma\in H^0(\End(V\oplus
W)\otimes L)$ is
not nilpotent, contrary to the case of triples.

\begin{proposition}
\label{vanishing main}
Let $E=(V,W,\beta,\gamma)$ be a $L$-twisted $\U(p,q)$-Higgs bundle and
$E'=(V',W',\beta',\gamma')$ a $L$-twisted $\U(p',q')$-Higgs
bundle. Assume that $E$ and $E'$ are $\alpha$-semistable with
$\mu_\alpha(E)=\mu_\alpha(E')$. Let $E''=\sigma(E')\otimes
L^{-1}K$. Assume that one of the following hypotheses hold:
\begin{itemize}
\item[(A)] $\deg(L)>2g-2$;
\item[(B)] $\deg(L)=2g-2$, both $E$ and $E'$ are $\alpha$-stable and
  there is no isomorphism $f\colon E \xra{\cong} E''$.
\end{itemize}
Then $\mathbb{H}^2(\mathpzc{Hom}^\bullet(E',E))=0$ if one of the
following additional conditions holds:
\begin{itemize}
\item[(1)] $\alpha=0$;
\item[(2)] $\alpha > 0$ and either $\beta'$ is
  injective or $\beta$ is surjective;
\item[(3)] $\alpha < 0$ and either $\gamma'$ is
  injective or $\gamma$ is surjective.
\end{itemize}
\end{proposition}

\begin{proof}
Suppose first that $\alpha=0$. Then either of the conditions (A) and
(B) guarantee that strict inequality holds in
\eqref{eq:h2-vanishing-ineq}. Hence Lemmas~\ref{lem:H2-is-H0} and
\ref{lem:lambda-f} imply the stated vanishing of $\HH^2$.

Now suppose that $\beta'\colon W'\to V'\otimes L$ is injective. If
$f\colon E\to E''$ is non-zero then, since $f$ is a morphism of
twisted $\U(p,q)$-Higgs bundles, we have $\rk(f(W)) \geq \rk(f(V))$.
Hence $\lambda(f) = \frac{\rk(f(V))}{\rk(f(V))+\rk(f(W))}$ satisfies
$\lambda(f)\leq 1/2$. If additionally $\alpha>0$, it follows that
$\alpha(2\lambda(f)-1) \leq 0$ which contradicts
Lemma~\ref{lem:lambda-f} under either of the conditions (A) and
(B). Therefore there are no non-zero morphisms $f\colon E\to E''$ and so  Lemma~\ref{lem:H2-is-H0} implies vanishing of $\HH^2(\mathpzc{Hom}^\bullet(E',E))$.

We have deduced vanishing of $\HH^2$ under the conditions $\alpha>0$
and $\beta'$ injective. The remaining conditions in (2) and (3) for
vanishing of $\HH^2$ can now be deduced by using symmetry arguments as
follows.

Suppose first that $\alpha<0$ and $\gamma'$ is injective. 
Then, using
Proposition~\ref{prop:sigma-alpha-stability}, $\sigma(E)$
is an $-\alpha$-semistable $\U(p,q)$-Higgs bundle and similarly for
$\sigma(E')$. Moreover, the $\beta$-map (which is
$\sigma(\gamma')$) of $\sigma(E')$ is injective.
Observe that
\begin{displaymath}
  \mathpzc{Hom}^\bullet(\sigma(E'),\sigma(E)) \cong
  \mathpzc{Hom}^\bullet(E',E).
\end{displaymath}
Hence, noting that $-\alpha>0$, the conclusion
follows from the previous case applied to the pair $(\sigma(E'),\sigma(E))$.

Next suppose that $\alpha<0$ and $\gamma$ is surjective. 
Then the dual $\U(p,q)$-Higgs bundle $E^*$ is $-\alpha$-semistable,
and similarly for $E'^*$. Moreover, the $\beta$-map (which is
$\gamma^*$) of $E^*$ is injective. Observe that
\begin{displaymath}
  \mathpzc{Hom}^\bullet(E^*,E'^*) \cong
  \mathpzc{Hom}^\bullet(E',E).
\end{displaymath}
Hence again the conclusion follows from the previous case, applied to
the pair $(E^*,E'^*)$.

The final case, $\alpha>0$ and $\beta$ surjective, follows in a
similar way, combining the two previous constructions.
\end{proof}

In the case when $q=1$ we can improve on 
Proposition~\ref{vanishing main}, as follows.
\begin{proposition}\label{q=1}
  Let $E$ be an $\alpha$-semistable $L$-twisted $\U(p,1)$-Higgs bundle
  with $p\geq 2$. Assume that $\deg(L)>2g-2$. Then
  $\mathbb{H}^2(\mathpzc{End}^\bullet(E))=0$ for all $\alpha$ in the
  range
\begin{displaymath}
   p(\mu(V)-\mu(W)) - (p+1)(\deg(L)-2g+2) < \alpha < p(\mu(V)-\mu(W)) + (p+1)(\deg(L)-2g+2).
\end{displaymath}
\end{proposition}

\begin{proof}
  Assume first that $\alpha\geq 0$. Note that an isomorphism as in (B) of the
  hypothesis of Proposition~\ref{vanishing main} cannot exist when
  $p\neq q$. Hence the proposition immediate gives the result if
  $\alpha=0$. Moreover, if $\beta\neq 0$, then it is injective, and
  hence $\mathbb{H}^2(\mathpzc{Hom}^\bullet(E',E))=0$ by (2) of the
  proposition.  We may thus assume that $\beta=0$ and consider the
  $L$-twisted triple $E_T\colon \gamma\colon V\to W\otimes L$. We have
  that
  \begin{displaymath}
    \HH^2(\mathpzc{End}^\bullet(E)) =
    \HH^2(\mathpzc{End}^\bullet(E_T)) \oplus H^1(\Hom(W,V)\otimes L),
  \end{displaymath}
  where $\mathpzc{End}^\bullet(E_T)$ is the deformation complex of the
  triple. The vanishing of $\HH^2(\mathpzc{End}^\bullet(E_T))$ for an
  $\alpha$-semistable triple when $\alpha>0$ is well
  known\footnote{Note that the stability parameter for the
    corresponding untwisted triple as considered in
    \cite{bradlow-garcia-prada-gothen:2004} is
    $\alpha+\deg(L)$.}
  (cf.~\cite{bradlow-garcia-prada-gothen:2004}). Hence it remains to
  show that $H^1(\Hom(W,V)\otimes L) = 0$ which, by Serre duality, is
  equivalent to the vanishing
  \begin{displaymath}
    H^0(\Hom(V,W)\otimes L^{-1}K) = 0.
  \end{displaymath}
  So assume we have a non-zero $f\colon V \to W\otimes L^{-1}K$. Then
  $f$ induces as non-zero map of line bundles $f\colon V/\ker(f) \to
  W\otimes L^{-1}K$ and hence
  \begin{equation}
    \label{eq:7}
    \deg(W) - \deg(L) + 2g-2 \geq \deg(V) - \deg(\ker(f)).
  \end{equation}
  On the other hand, since $\beta=0$ we can consider the subobject
  $(\ker(f),W,0,\gamma)$ of $E$ and hence, by $\alpha$-semistability,
  \begin{align}
    \mu_\alpha(\ker(f)\oplus W) &\leq \mu_\alpha(V \oplus W) \notag \\
    \iff (p+1)\deg(\ker(f)) + \deg(W) &\leq p\deg(V)+\alpha,
    \label{eq:6}                                      
  \end{align}
  where we have used that $\rk(\ker(f)) = p-1$ and $\rk(W)=1$. 
  Now combining \eqref{eq:7} and \eqref{eq:6} we obtain
  \begin{displaymath}
    \alpha \geq p(\mu(V)-\mu(W)) + (p+1)(\deg(L)-2g+2).
  \end{displaymath}
  This establishes the vanishing of $\HH^2$ for $\alpha$ in the range
  \begin{displaymath}
    0\leq \alpha < p(\mu(V)-\mu(W)) + (p+1)(\deg(L)-2g+2).
  \end{displaymath}
  On the other hand, if $\alpha\leq 0$, applying the preceding result
  to the dual twisted $\U(p,q)$-Higgs bundle
  $(V^*,W^*,\gamma^*,\beta^*)$ gives vanishing of $\HH^2$ for $\alpha$
  in the range 
  \begin{displaymath}
    0\geq \alpha > p(\mu(V)-\mu(W)) - (p+1)(\deg(L)-2g+2).
  \end{displaymath}
  This finishes the proof.
\end{proof}

In general the preceding proposition does not guarantee vanishing of
$\HH^2$ for all values of the parameter $\alpha$. But for some values
of the topological invariants, the upper bound of the preceding
proposition is actually larger than the maximal value for the
parameter $\alpha$. More precisely, we have the following result.

\begin{proposition}
  \label{prop:q=1-range}
  Let $E$ be an $\alpha$-semistable $L$-twisted $\U(p,1)$-Higgs bundle
  with $p\geq 2$. Assume that $\deg(L)>2g-2$. We have the following:
  \begin{itemize}
\item[(1)] 
    If $p\bigl(\mu(V)-\mu(W)\bigr) > 2g-2-(p-2)\bigl(\deg(L)-(2g-2)\bigr)$ then $\mathbb{H}^2(\mathpzc{End}^\bullet(E))=0$ for all $\alpha\geq 0$\\
    \item[(2)] If $p\bigl(\mu(V)-\mu(W)\bigr) < -2g+2+(p-2)\bigl(\deg(L)-(2g-2)\bigr)$ then $\mathbb{H}^2(\mathpzc{End}^\bullet(E))=0$ for all $\alpha\leq 0$
      \end{itemize}
\end{proposition}
\begin{proof}
  The upper and lower bound for $\alpha$ given in Proposition~\ref{bound for
    alpha} is, in this case
  \begin{displaymath}
    \alpha_M = -\frac{2p}{p-1}\bigl(\mu(V)-\mu(W)\bigr)
    +\frac{p+1}{p-1}\deg(L),
  \end{displaymath}
  \begin{displaymath}
    \alpha_m = -\frac{2p}{p-1}\bigl(\mu(V)-\mu(W)\bigr)
    -\frac{p+1}{p-1}\deg(L).
  \end{displaymath}
  It is simple to check that the inequalities of the statements are
  equivalent to $\alpha_M$ being less than the upper bound and $\alpha_m$ being bigger than the lower bound for
  $\alpha$ of Proposition~\ref{q=1}.
\end{proof}



The following trivial observation is sometimes useful. 

\begin{proposition}
  \label{prop:end-hom-vanishing}
  Let $E$ and $E'$ be $L$-twisted $\U(p,q)$-Higgs bundles such that
  $\HH^2(\mathpzc{End}^\bullet(E\oplus E'))=0$. Then
  \begin{displaymath}
    \mathbb{H}^2(\mathpzc{Hom}^\bullet(E',E)) =
    \mathbb{H}^2(\mathpzc{Hom}^\bullet(E,E')) = 0.
  \end{displaymath}
\end{proposition}
\begin{proof}
  Immediate in view of Remark~\ref{rem:direct-sum-deform}.
\end{proof}

We can summarize our main results on vanishing of $\HH^2$ as follows.

\begin{lemma}
  \label{lem:h2-vanishing-final}
Fix a type $t=(p,q,a,b)$ and let $E$ be an $\alpha$-semistable $L$-twisted
$\U(p,q)$-Higgs bundle of type $t$ with $\deg(L)\geq 2g-2$. If
$\deg(L)=2g-2$ assume moreover that $E$ is $\alpha$-stable.
If either one of the following conditions holds:
\begin{itemize}
\item[$(1)$] $q=1$, $p\geq 2$ and $p(a/p-b/q) - \deg(L)(p+1) < \alpha < p(a/p-b/q) + \deg(L)(p+1)$,
\item[$(2)$] $a/p-b/q>-\deg(L)$ and $0\leq\alpha<\frac{2pq}{\mathrm{min}\{p,q\}|p-q|+p+q}\big(b/q-a/p-\deg(L)\big)+\deg(L)$,
\item[$(3)$] $a/p-b/q<\deg(L)$ and $\frac{2pq}{\mathrm{min}\{p,q\}|p-q|+p+q}(b/q-a/p+\deg(L))-\deg(L)<\alpha\leq 0$.
\end{itemize}
Then $\mathbb{H}^2(\mathpzc{End}^\bullet(E))$ vanishes.
\end{lemma}

\begin{proof}
For part $(1)$, use Proposition~\ref{q=1}. The other parts  follow from
Proposition~\ref{surjective}, Corollary~\ref{Cor:Surjective}, and Proposition~\ref{vanishing main}.
\end{proof}

\subsection{Moduli space of twisted $\U(p,q)$-Higgs bundles}
\label{sec:moduli-space-twisted}

Finally, we are in a position to make statements about smoothness of
the moduli space.  Recall that we denote
the moduli space of $\alpha$-polystable twisted $\U(p,q)$-Higgs
bundles with type $t=(p,q,a,b)$ by
$$
\mathcal{M}_\alpha(t)=\mathcal{M}_\alpha(p,q,a,b),
$$ 
and the moduli space of $\alpha$-stable twisted $\U(p,q)$-Higgs bundle
by $\mathcal{M}_\alpha^s(t)\subset\mathcal{M}_\alpha(t)$.

\begin{proposition}\label{smooth}
Fix a type $t=(p,q,a,b)$. 
If either one of the following conditions holds:
\begin{itemize}
\item[$(1)$] $q=1$, $p\geq 2$ and $p(a/p-b/q) - \deg(L)(p+1) < \alpha < p(a/p-b/q) + \deg(L)(p+1)$,
\item[$(2)$] $a/p-b/q>-\deg(L)$ and $0\leq\alpha<\frac{2pq}{\mathrm{min}\{p,q\}|p-q|+p+q}\big(b/q-a/p-\deg(L)\big)+\deg(L)$,
\item[$(3)$] $a/p-b/q<\deg(L)$ and $\frac{2pq}{\mathrm{min}\{p,q\}|p-q|+p+q}(b/q-a/p+\deg(L))-\deg(L)<\alpha\leq 0$.
\end{itemize}
Then the moduli space $\mathcal{M}_\alpha^s(t)$ is smooth. 
\end{proposition}

\begin{proof}
Combine Lemma~\ref{lem:h2-vanishing-final} and Proposition~\ref{smooth0}. 
\end{proof}

\section{Crossing critical values}
\label{sec:crossing}

\subsection{Flip loci}\label{Flip loci}
In this section we study the variation with $\alpha$ of the moduli
spaces $\mathcal{M}_\alpha^s(t)$ for fixed type $t=(p,q,a,b)$. We are
using a method similar to the one for chains given in
\cite{Schmitt:2006}, which in turn is based on
\cite{bradlow-garcia-prada-gothen:2004}.

Let $\alpha_c$ be a critical value. We adopt the following
notation:
$$
\alpha_c^+=\alpha_c+\epsilon,\mbox{ }
\alpha_c^-=\alpha_c-\epsilon,
$$
where $\epsilon>0$ is small enough so that $\alpha_c$ is the only
critical value in the interval $(\alpha_c^-, \alpha_c^+)$. We begin
with a set theoretic description of the differences between two spaces
$\mathcal{M}_{\alpha_c^+}$ and $\mathcal{M}_{\alpha_c^-}$.

\begin{definition}
We define \emph{flip loci} $\mathcal{S}_{\alpha_c^{\pm}}
\subset\mathcal{M}^s_{\alpha_c^{\pm}}$ by the condition that the
points in $\mathcal{S}_{\alpha_c^+}$ represent twisted $\U(p,q)$-Higgs
bundles which are $\alpha_c^+$-stable but $\alpha_c^-$-unstable, and
analogously for $\mathcal{S}_{\alpha_c^-}$.
\end{definition}


A twisted $\U(p,q)$-Higgs bundle $E\in\mathcal{S}_{\alpha_c^{\pm}}$ is
strictly $\alpha_c$-semistable and so we can use the Jordan-H\"{o}lder filtrations of $E$ in order to estimate the codimension of $\mathcal{S}_{\alpha_c^{\pm}}$ in $\mathcal{M}_{\alpha_c^{\pm}}$. 
The following is an analogue for twisted $\U(p,q)$-Higgs bundles of
\cite[Proposition~4.3]{Schmitt:2006}, which is a result for chains.
\begin{proposition}\label{Prop:bound-for-dim}
Fix a type $t=(p,q,a,b)$. Let $\alpha_c$ be a critical value and let
$\mathcal{S}$ be a family of $\alpha_c$-semistable twisted
$\U(p,q)$-Higgs bundles $E$ of type $t$, all of them pairwise
non-isomorphic, and whose Jordan-H\"{o}lder filtrations has an
associated graded of the form $\mathrm{Gr}(E)=\bigoplus_{i=1}^mQ_i$, with $Q_i$ twisted $\U(p,q)$-Higgs bundle of type $t_i$. 
If either one of the following conditions holds:
\begin{itemize}
\item[$(1)$] $q=1$, $p\geq 2$ and $p(a/p-b/q) - \deg(L)(p+1) < \alpha_c < p(a/p-b/q) + \deg(L)(p+1)$,
\item[$(2)$] $a/p-b/q>-\deg(L)$ and $0\leq\alpha_c<\frac{2pq}{\mathrm{min}\{p,q\}|p-q|+p+q}\big(b/q-a/p-\deg(L)\big)+\deg(L)$,
\item[$(3)$] $a/p-b/q<\deg(L)$ and $\frac{2pq}{\mathrm{min}\{p,q\}|p-q|+p+q}(b/q-a/p+\deg(L))-\deg(L)<\alpha_c\leq 0$.
\end{itemize}
Then
\begin{equation}
  \label{bund}
  \dim \mathcal{S}\leq -\sum_{i\leq j}\chi(t_j,t_i)-\frac{m(m-3)}{2}.
\end{equation}
\end{proposition}

\begin{proof}
Once appropriate vanishing of $\HH^2$ is ensured, the proof is similar
to the proof of \cite[Proposition~4.3]{Schmitt:2006};
we indicate the idea for $m=2$. In view of the definition of $\mathcal{S}$, there exists an injective canonical map $$i: \mathcal{S}\to\mathcal{M}^s_{\alpha_c}(t_1)\times \mathcal{M}^s_{\alpha_c}(t_2)$$
with $i^{-1}(Q_1,Q_2)\cong \mathbb{P}(\Ext^1(Q_2,Q_1))$, where $\mathbb{P}(\Ext^1(Q_2,Q_1))$ parametrizes equivalence classes of extensions 
$$0\rightarrow Q_1\rightarrow E \rightarrow Q_2\rightarrow 0.$$ Notice
that $Q_1$ and $Q_2$ satisfy the hypothesis of Proposition
$\ref{vanishing main}$ (or, in case $q=1$, Proposition~\ref{q=1}; cf.\ Proposition~\ref{prop:end-hom-vanishing}) and therefore, cf.\ Proposition $\ref{chi}$, $\dim(\mathbb{P}\Ext^1(Q_2,Q_1))$ is constant as $Q_1$ and $Q_2$ vary in their corresponding moduli spaces.  
 Hence, we
 obtain $$\dim\mathcal{S}\leq\dim\mathcal{M}_{\alpha_c}^s(t_1)+\dim\mathcal{M}_{\alpha_c}^s(t_2)+\dim\mathbb{P}(\Ext^1(Q_2,Q_1)).$$ 
The general case follows by induction on $m$ as in loc.~cit.
\end{proof}

In order to show that the flip loci $\mathcal{S}_{\alpha_c^{\pm}}$ has
positive codimension we need to bound the values of $\chi(t_i,t_j)$ in
$(\ref{bund})$. This is what we do next.

\subsection{Bound for $\chi$}
\label{sec:bound-chi}
Here we consider a $Q$-bundle associated to the complex $\mathpzc{Hom}^\bullet(E',E)$ and construct a solution to the vortex equations on this $Q$-bundle from solutions on $E'$ and $E$. The quiver $Q$ is the following:
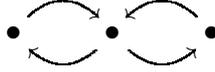
\[
\xymatrix{
\bullet\ar@{<-}@/_1pc/[r]&\bullet\ar@{<-}@/_1pc/[l]\ar@{>}@/_1pc/[r]&\bullet\ar@{>}@/_1pc/[l]\\
}
\]
\\
The construction generalizes the one of \cite{bradlow-garcia-prada-gothen:2004} Lemma $4.2$. 
\subsubsection{The $Q$-bundle associated to $\mathpzc{Hom}^\bullet(E',E)$}
Let  $E=(V,W,\beta,\gamma)$ be a $L$-twisted $\U(p,q)$-Higgs bundle and $E'=(V',W',\beta',\gamma')$ a $L$-twisted $\U(p',q')$-Higgs bundle.
Let us consider the following twisted $Q$-bundle $\widetilde{E}$ (the morphisms are twisted by $L$ for each arrow):
\begin{equation}\label{associated}
\xymatrix{
\Hom(W',V)\ar@{<-}@/_2pc/[r]^{\phi_b}&\Hom(V',V)\oplus \Hom(W',W)\ar@{<-}@/_2pc/[l]_{\phi_d} \ar@{>}@/_2pc/[r]^{\phi_a}&\Hom(V',W)\ar@{>}@/_2pc/[l]_{\phi_c}\\
}
\end{equation}
where
\begin{eqnarray*}
\phi_a(f_1,f_2)&=&(f_2\otimes 1_L)\circ \gamma'-\gamma\circ f_1\mbox{ },\\
\phi_b(f_1,f_2)&=&(f_1\otimes 1_L)\circ \beta' -\beta\circ f_2\mbox{ },\\
\phi_c(g)&=&(\beta\circ g, (g\otimes 1_L)\circ \beta')\mbox{ },\\
\phi_d(h)&=&((h\otimes 1_L)\circ\gamma',\gamma\circ h).\\
\end{eqnarray*} 
We will write briefly as $\widetilde{E}$
\[
\xymatrix{
\mathpzc{Hom}^{12}\ar@{<-}@/_1pc/[r]^{\phi_b}_L&\mathpzc{Hom}^0\ar@{<-}@/_1pc/[l]_{\phi_d}^L \ar@{>}@/_1pc/[r]^{\phi_a}_L&\mathpzc{Hom}^{11}\ar@{>}@/_1pc/[l]_{\phi_c}^L\\
}.
\]
Note that $\mathpzc{Hom}^1=\mathpzc{Hom}^{11}\oplus
\mathpzc{Hom}^{12}$ and $a_0=(\phi_a,\phi_b)$, where $a_0:
\mathpzc{Hom}^0\to \mathpzc{Hom}^1$ is the Hom-complex
$(\ref{deformation complex})$. 

In this section, by using Proposition $\ref{solution}$, we prove that
if $E'$ and $E$ are $\alpha$-polystable then $\widetilde{E}$ is
$\boldsymbol{\alpha}$-polystable for a suitable choice of
$\boldsymbol{\alpha}$.

\begin{lemma}\label{vortex}
Let  $E=(V,W,\beta,\gamma)$ be a $L$-twisted $\U(p,q)$-Higgs bundle
and $E'=(V',W',\beta',\gamma')$ a $L$-twisted $\U(p',q')$-Higgs
bundle. Suppose, moreover, we have solutions to the $(\tau_1,\tau_2)$-vortex equations on $E$ and the $(\tau_1',\tau_2')$-vortex equations on $E'$ such that $\tau_1-\tau'_1=\tau_2-\tau'_2$. 
Then  the induced Hermitian metric on the $Q$-bundle $\widetilde{E}$ satisfies the vortex equations
\begin{eqnarray*}
\sqrt{-1}\Lambda F(\mathpzc{Hom}^{12})+\phi_b\phi_b^{\ast}-\phi_d^{\ast}\phi_d&=&\widetilde{\tau_2} \mathrm{Id}_{\mathpzc{Hom}^{12}},\\
\sqrt{-1}\Lambda F(\mathpzc{Hom}^0)+\phi_c\phi_c^{\ast}+\phi_d\phi_d^{\ast}-\phi_a^{\ast}\phi_a-\phi_b^{\ast}\phi_b&=&\widetilde{\tau_1} \mathrm{Id}_{\mathpzc{Hom}^0},\\
\sqrt{-1}\Lambda F(\mathpzc{Hom}^{11})+\phi_a\phi_a^{\ast}-\phi_c^{\ast}\phi_c&=&\widetilde{\tau_0}  \mathrm{Id}_{\mathpzc{Hom}^{11}}.
\end{eqnarray*}
For $\boldsymbol{\tau}=\big(\widetilde{\tau_0},\widetilde{\tau_1},\widetilde{\tau_2}\big)$ given by
\begin{eqnarray*}
\widetilde{\tau_0}&=&\tau_2-\tau'_1,\\
 \widetilde{\tau_1}&=&\tau_1-\tau_1'=\tau_2-\tau'_2, \\
 \widetilde{\tau_2}&=&\tau_1-\tau'_2.
\end{eqnarray*}
\end{lemma}
\begin{proof}
The vortex equations for $E$ and $E'$ are
\begin{align*}
\sqrt{-1}\Lambda F(V)+\beta\beta^\ast-\gamma^\ast\gamma
&=\tau_1 \mathrm{Id}_{V},\\
\sqrt{-1}\Lambda F(W)+\gamma\gamma^\ast-\beta^\ast\beta
&=\tau_2 \mathrm{Id}_{W},\\
\sqrt{-1}\Lambda F(V')+\beta'\beta'^\ast-\gamma'^\ast\gamma'
&=\tau_1' \mathrm{Id}_{V'},\\
\sqrt{-1}\Lambda F(W')+\gamma'\gamma'^\ast-\beta'^\ast\beta'
&=\tau_2'\mathrm{Id}_{W'}.
\end{align*}
We have
\begin{eqnarray*}
F(\mathpzc{Hom}^0)(\psi,\eta)&=&(F(V)\circ\psi-\psi\circ F(V'),F(W)\circ\eta-\eta\circ F(W')).
\end{eqnarray*}
Now we calculate $\phi_a^\ast$ and $\phi_b^\ast$: for $(f_1,f_2)\in \mathpzc{Hom}^0$, $g\in \mathpzc{Hom}^{11}$ and $h\in \mathpzc{Hom}^{12}$  we have,
\begin{eqnarray*}
\lefteqn{\left\langle  \phi_a^\ast(g),(f_1, f_2) \right\rangle _{\mathpzc{Hom}^0}}\\
&=&\left\langle g, \phi_a\big((f_1, f_2)\big) \right\rangle_{\mathpzc{Hom}^{11}}\\
&=& \left\langle g,(f_2\otimes 1_L)\circ\gamma'-\gamma\circ f_1\right\rangle_{\mathpzc{Hom}^{11}}\\
&=&\left\langle g,(f_2\otimes 1_L)\circ\gamma'\right\rangle_{C_{11}} - \left\langle g,\gamma\circ f_1\right\rangle_{\mathpzc{Hom}^{11}}\\
&=&\left\langle (g\circ \gamma'^\ast)\otimes 1_{L^\ast},f_2\right\rangle_{\Hom(W',W)} + \left\langle -\gamma^\ast\circ g,f_1\right\rangle_{\Hom(V',V)}\\
&=&\left\langle  \big(-\gamma^\ast\circ g, (g\circ\gamma'^\ast)\otimes1_{L^\ast}\big),(f_1, f_2) \right\rangle _{\mathpzc{Hom}^0}\\
\end{eqnarray*}
and
\begin{eqnarray*}
\left\langle  \phi_b^\ast(h),(f_1, f_2) \right\rangle _{\mathpzc{Hom}^0}&=&\left\langle h, \phi_b\big((f_1, f_2)\big) \right\rangle_{\mathpzc{Hom}^{12}}\\
&=& \left\langle h,(f_1\otimes 1_L)\circ\beta'-\beta\circ f_2\right\rangle_{\mathpzc{Hom}^{12}}\\
&=&\left\langle (h\circ\beta'^\ast)\otimes 1_{L^\ast},f_1\right\rangle_{\Hom(V',V)}- \left\langle \beta^\ast\circ h,f_2\right\rangle_{\Hom(W',W)}\\
&=&\left\langle \big((h\circ\beta'^\ast)\otimes 1_{L^\ast},-\beta^\ast\circ h\big),\big(f_1,f_2\big)\right\rangle_{\mathpzc{Hom}^0}\\
\end{eqnarray*}

Hence,
\begin{eqnarray*}
\phi_a^\ast(g)&=&(-\gamma^\ast\circ g,(g\circ\gamma'^\ast)\otimes1_{L^\ast}),\\
\phi_b^\ast(h)&=&((h\circ\beta'^\ast)\otimes 1_{L^\ast},-\beta^\ast\circ h).
\end{eqnarray*}
By a similar calculation as above, we have
\begin{eqnarray*}
\phi_c^\ast(f_1, f_2)&=&(f_2\circ\beta'^\ast)\otimes 1_{L^\ast}-\beta^\ast\circ f_1,\\
\phi_d^\ast(f_1, f_2)&=& (f_1\circ\gamma'^\ast)\otimes1_{L^\ast}-\gamma^\ast\circ f_2.
\end{eqnarray*}
Let $g\in \mathpzc{Hom}^{11} $ and $h\in \mathpzc{Hom}^{12}$, then we have:
\begin{eqnarray*}
\phi_c^\ast\phi_c(g)&=&\phi_c^\ast(\beta\circ g,(g\otimes1_L)\circ\beta')\\
&=&\beta^\ast\beta\circ g+g\circ\beta'\beta'^\ast.
\end{eqnarray*}
\begin{eqnarray*}
\phi_d^\ast \phi_d(h)&=&\phi_d^\ast\big((h\otimes 1_L)\circ\gamma',\gamma\circ h\big)\\
&=&h\circ\gamma'\gamma'^\ast-\gamma^\ast\gamma\circ h.
\end{eqnarray*}
and
\begin{eqnarray*}
\phi_b\phi_b^\ast(h)&=&\phi_b(h\circ\beta'^\ast\otimes 1_{L^\ast},\beta^\ast\circ h)\\
&=&h\circ\beta'^\ast\beta'-\beta\beta^\ast.
\end{eqnarray*}
\begin{eqnarray*}
\phi_a\phi_a^\ast(g)&=&\phi_a(g\circ\gamma'^\ast\otimes1_{L^\ast},-\gamma^\ast\circ g)\\
&=&g\circ\gamma'^\ast\gamma'+\gamma\gamma^\ast\circ g.
\end{eqnarray*}

Thus,
\begin{eqnarray*}
\phi_b\phi_b^\ast-\phi_d^\ast \phi_d(h)&=&h\circ\beta'^\ast\beta'-\beta\beta^\ast\circ h-h\circ\gamma'\gamma'^\ast+\gamma^\ast\gamma\circ h\\
\phi_a\phi_a^\ast-\phi_c^\ast \phi_c(g)&=&g\circ\gamma'^\ast\gamma'+\gamma\gamma^\ast\circ g-\beta^\ast\beta\circ g-g\circ\beta'\beta'^\ast
\end{eqnarray*}

Hence for $g\in \mathpzc{Hom}^{11}$ and $h\in \mathpzc{Hom}^{12}$ we have,

\begin{eqnarray*}
\lefteqn{(\sqrt{-1}\Lambda F(\mathpzc{Hom}^{11})+\phi_a\phi_a^\ast-\phi_c^\ast \phi_c)(g)}\\
&=&\sqrt{-1}\Lambda\big(F(W)\circ g-g\circ F(V')\big)+\phi_a\phi_a^\ast-\phi_c^\ast \phi_c(g)\\
&=&\big(\sqrt{-1}\Lambda F(W)+\gamma\gamma^\ast-\beta^\ast\beta\big)\circ g+g\circ\big(-\sqrt{-1}\Lambda F(V')+\gamma'^\ast\gamma'-\beta'\beta'^\ast\big)\\
&=&\tau_2\mathrm{Id}_{W}\circ g-g\circ\tau'_1\mathrm{Id}_{V'}\\
&=&(\tau_2-\tau'_1)g
\end{eqnarray*}
\begin{eqnarray*}
\lefteqn{(\sqrt{-1}\Lambda F(\mathpzc{Hom}^{12})+\phi_b\phi_b^\ast-\phi_d^\ast \phi_d)(h)}\\
&=&\sqrt{-1}\Lambda\big( \otimes F(V)\circ h-h\circ F(W')\big)+\phi_b\phi_b^\ast-\phi_d^\ast \phi_d(h)\\
& = &\big(\sqrt{-1}\Lambda F(V)+\gamma^\ast\gamma-\beta\beta^\ast\big)\circ h+h\circ\big( -\sqrt{-1}\Lambda F(W')+\beta'^\ast\beta'-\gamma'\gamma'^\ast\big)\\
&=&\tau_1\mathrm{Id}_{V}\circ h-h\circ\tau'_2\mathrm{Id}_{W'}\\
&=&(\tau_1-\tau'_2)h.
\end{eqnarray*}
Similarly for $(f_1,f_2)\in \mathpzc{Hom}^0$ we have,
\begin{eqnarray*}
\phi_c\phi_c^\ast (f_1,f_2)&=&\phi_c((f_2\circ\beta'^\ast)\otimes1_{L^\ast}-\beta^\ast\circ f_1)\\
&=&\big(\beta\beta^\ast\circ f_1-\beta\circ(f_2\circ\beta'^\ast\otimes1_{L^\ast}),f_2\circ\beta'^\ast\beta'-(\beta^\ast\circ f_1\otimes 1_{L})\otimes \beta'\big)
\end{eqnarray*}
\begin{eqnarray*}
\phi_d\phi_d^\ast (f_1,f_2)&=&\phi_d\big((f_1\circ\gamma'^\ast)\otimes1_{L^\ast}-\gamma^\ast\circ f_2\big)\\
&=&\left( f_1\circ\gamma'^\ast\gamma'-\gamma^\ast\circ f_2\otimes1_L\circ\gamma',\gamma\circ(f_1\circ\gamma'^\ast\otimes1_{L^\ast})-\gamma\gamma^\ast\circ f_2\right)
\end{eqnarray*}
and
\begin{eqnarray*}
\phi_a^\ast\phi_a (f_1,f_2)&=&\phi_a^\ast\big(f_2\otimes1_L\circ\gamma'-\gamma\circ f_1\big)\\
&=&\left(-\gamma^\ast\circ f_2\otimes1_L\circ\gamma'+\gamma^\ast\gamma\circ f_1,(f_2\circ\gamma'\gamma'^\ast-\gamma\circ f_1\circ\gamma'^\ast\otimes1_{L^\ast}\right)
\end{eqnarray*}
\begin{eqnarray*}
\phi_b^\ast \phi_b(f_1,f_2)&=&\phi_b^\ast\big(f_1\otimes1_L\circ\beta'-\beta\circ f_2\big)\\
&=&\left(f_1\circ\beta'\beta'^\ast-\beta\circ f_2\circ\beta'^\ast\otimes1_{L^\ast},\beta^\ast\circ f_1\otimes1_L\circ\beta'-\beta^\ast\beta\circ f_2\right)
\end{eqnarray*}
So,
\begin{eqnarray*}
\lefteqn{(\phi_c\phi_c^\ast+\phi_d\phi_d^\ast-\phi_a^\ast\phi_a-\phi_b^\ast\phi_b)(f_1,f_2)}\\
& =&\left(\beta\beta^\ast\circ f_1+f_1\circ\gamma'^\ast\gamma'-\gamma^\ast\gamma\circ f_1-f_1\circ\beta'\beta'^\ast,f_2\circ\beta'^\ast\beta'-\gamma\gamma^\ast\circ f_2-f_2\circ\gamma'\gamma'^\ast+\beta^\ast\beta\circ f_2\right)
\end{eqnarray*}
Hence we have,
\begin{eqnarray*}
\lefteqn{(\sqrt{-1}\Lambda F(\mathpzc{Hom}^0)+\phi_c\phi_c^\ast+\phi_d\phi_d^\ast-\phi_a^\ast\phi_a-\phi_b^\ast\phi_b)(f_1,f_2)}\\
&=&\Big(\sqrt{-1}\Lambda (F(V)\circ f_1-f_1\circ  F(V')),\sqrt{-1}\Lambda( F(W)\circ f_2-f_2\circ  F(W'))\Big)+\\
& &\Big(\beta\beta^\ast\circ f_1+f_1\circ\gamma'^\ast\gamma'-\gamma^\ast\gamma\circ f_1-f_1\circ\beta'\beta'^\ast,f_2\circ\beta'^\ast\beta'-\gamma\gamma^\ast\circ f_2-f_2\circ\gamma'\gamma'^\ast+\beta^\ast\beta\circ f_2\Big)\\
&=&\Big((\sqrt{-1}\Lambda F(V)+\beta\beta^\ast-\gamma^\ast\gamma)\circ f_1+f_1\circ(-\sqrt{-1}\Lambda F(V')+\gamma'^\ast\gamma'-\beta'\beta'^\ast),\\
& &(\sqrt{-1}\Lambda F(W)+\gamma\gamma^\ast-\beta^\ast\beta)\circ f_2+f_2\circ(-\sqrt{-1}\Lambda F(W')+\beta'^\ast\beta'-\gamma'\gamma'^\ast)\Big)\\
&=&\Big((\tau_1-\tau'_1)f_1,(\tau_2-\tau'_2)f_2\Big).
\end{eqnarray*}
The proof is completed, since by assumption $\tau_1-\tau'_1=\tau_2-\tau'_2$. 
\end{proof}

\begin{theorem}\label{Polystable}
Let  $E=(V,W,\beta,\gamma)$ be a $L$-twisted $\U(p,q)$-Higgs bundle
and $E'=(V',W',\beta',\gamma')$ a $L$-twisted $\U(p',q')$-Higgs
bundle. Then the $Q$-bundle $\widetilde{E}$ is $\boldsymbol{\alpha}$-polystable for $\boldsymbol{\alpha}=(\alpha,2\alpha)$ .
\end{theorem}

\begin{proof}
Since $E$ and $E'$ are $\alpha$-polystable, from Theorem $\ref{solution}$ follows that they support solutions to the $(\tau_1,\tau_2)$- and $(\tau_1',\tau_2')$-vortex equations where $\alpha=\tau_2-\tau_1=\tau_2'-\tau_1'$. Using Lemma $\ref{vortex}$ it follows that the $Q$-bundle $\widetilde{E}$ admits a Hermitian metric such that vortex equations are satisfied for $\boldsymbol{\tau}=(\tau_2-\tau_1',\tau_2-\tau'_2,\tau_1-\tau'_2)$. Now from Theorem $\ref{Hitchin-Kobayashi}$ we get that $\widetilde{E}$ is $\boldsymbol{\alpha}$-polystable for 
\begin{eqnarray*}
\alpha_1=\tau_2-\tau'_1-\tau_2+\tau'_2=\alpha,\\
\alpha_2=\tau_2-\tau'_1-\tau_1+\tau'_2=2\alpha.
\end{eqnarray*}
\end{proof}
\subsubsection{Bound for $\chi(E',E)$}
We are using the method in \cite{bradlow-garcia-prada-gothen:2004} and we start with some lemmas needed to estimate $\chi(E',E)$. 
\begin{lemma}
Let  $E=(V,W,\beta,\gamma)$ be a $L$-twisted $\U(p,q)$-Higgs bundle
and $E'=(V',W',\beta',\gamma')$ a $L$-twisted $\U(p',q')$-Higgs
bundle. Let $\mathpzc{Hom}^\bullet(E',E)$ be the deformation complex of $E$ and $E'$, as in $(\ref{deformation complex of U(p,q)})$. Then the following inequalities hold.
\begin{align}
\deg(\ker(a_0)&\leq \rk(\ker(a_0))\big(\mu_\alpha(E')-\mu_\alpha(E)\big)\label{kernelH}\\ 
\deg(\im(a_0)&\leq\big(\rk(\mathpzc{Hom}^{1})-\rk(\im(a_0))\big)\big(\mu_\alpha(E)-\mu_\alpha(E')-\deg(L)\big)-\label{cokernelH}\\
 & \alpha\big(\rk(\mathpzc{Hom}^1)-\rk(\im(a_0))
  -2\rk(\coker(\phi_b))\big)+\deg(\mathpzc{Hom}^{1}).\nonumber
\end{align}
\end{lemma}

\begin{proof}
Assume that $\rk(\ker(a_0))>0$ as if it is zero then $(\ref{kernelH})$ is obvious. It follows from Proposition $\ref{Polystable}$ that the $Q$-bundle $\widetilde{E}$ is $\boldsymbol{\alpha}=(\alpha,2\alpha)$-polystable. We can define a subobject of $\widetilde{E}$ by \\\\
\[
\xymatrix{
\mathcal{K}:\mbox{ } 0\ar@{<-}@/_2pc/[r]^{}&\ker(a_0)\ar@{<-}@/_2pc/[l]_{} \ar@{>}@/_2pc/[r]^{}&0\ar@{>}@/_2pc/[l]_{}.\\
}
\]
It follows from the $\boldsymbol{\alpha}$-polystability that
\begin{eqnarray*}
\mu_\alpha(\mathcal{K})=\mu(\ker(a_0))+\alpha&\leq& \mu_{\boldsymbol{\alpha}}(\widetilde{E})=\mu_\alpha(E')-\mu_\alpha(E)+\alpha.
\end{eqnarray*}
Thus we have
$$\mu(\ker(a_0))\leq\mu_\alpha(E')-\mu_\alpha(E),$$
which is equivalent to $(\ref{kernelH})$. The second inequality is obvious when $\rk(\im(a_0))=\rk(\mathpzc{Hom}^{1})$. We thus assume $\rk(\im(a_0))<\rk(\mathpzc{Hom}^1)$. We define a quotient of the bundle $\widetilde{E}$ by
\[
\xymatrix{
\mathcal{Q:}\mbox{      }\coker(\phi_b)\otimes L^{-1}\ar@{<-}@/_1pc/[r]^{}&0\ar@{<-}@/_1pc/[l]_{} \ar@{>}@/_1pc/[r]^{}&\coker(\phi_a)\otimes L^{-1}\ar@{>}@/_1pc/[l]_{}\\
}
\]
(we take the saturation if cokernels are not torsion free). By the $\boldsymbol{\alpha}$-polystability of $\widetilde{E}$ we have
\begin{equation}\label{Q}
\mu_{\boldsymbol{\alpha}}(\mathcal{Q})=\mu(\mathcal{Q})+2\alpha\frac{ \rk(\coker(\phi_b))}{\rk(\coker(\phi_a))+\rk(\coker(\phi_b))}\geq \mu_{\boldsymbol{\alpha}}(\widetilde{E})=\mu_\alpha(E')-\mu_\alpha(E)+\alpha. 
\end{equation}
Note that $\mu(\mathcal{Q})=\mu(\coker(a_0))-\deg(L)$. This and $(\ref{Q})$, together with the fact that $$\mu(\coker(a_0))\leq\frac{\deg(\mathpzc{Hom}^{1})-\deg(\im(a_0))}{\rk(\mathpzc{Hom}^{1})-\rk(\im(a_0))},$$
lead us to $(\ref{cokernelH})$.
\end{proof}

\begin{lemma}\label{bund for Euler}
  Let  $E=(V,W,\beta,\gamma)$ be a $L$-twisted $\U(p,q)$-Higgs bundle
and $E'=(V',W',\beta',\gamma')$ a $L$-twisted $\U(p',q')$-Higgs
bundle. Assume that $p-q$ and $p'-q'$ have the same sign, and suppose that the following conditions hold:
\begin{itemize}
\item $-\deg(L)\leq\alpha\leq\deg(L)$ and $\deg(L)\geq 2g-2$,
\item$E$ and $E'$ are $\alpha$-polystable with $\mu_\alpha(E)=\mu_\alpha(E')$,
\item the map $a_0$ is not an isomorphism.
\end{itemize}
 Then $$\chi(E',E)\leq1-g,$$ 
if the map $a_0$ is not generically an isomorphism, otherwise $\chi(E',E)<0$.
\end{lemma}

\begin{proof}
By the estimates $(\ref{kernelH})$ and $(\ref{cokernelH})$, we 
obtain
\begin{align*}
 \deg(\ker(a_0))+\deg(\im(a_0))&
  \leq\big(\mu_\alpha(E')-\mu_\alpha(E)\big)
  \Big(\rk(\ker(a_0))+\rk(\im(a_0))-\rk(\mathpzc{Hom}^{1})\Big)\\
&\phantom{\leq}-\alpha\big(\rk(\coker(\phi_a))-\rk(\coker(\phi_a))\big)\\
&\phantom{\leq}-\deg(L)\big(\rk(\mathpzc{Hom}^1)-\rk(\im(a_0))\big)+\deg(\mathpzc{Hom}^{1}).
\end{align*}
As $\mu_\alpha(E)=\mu_\alpha(E')$ we deduce 
\begin{multline*}
\deg(\mathpzc{Hom}^0)-\deg(\mathpzc{Hom}^{1})\\
\leq-\alpha\big(\rk(\coker(\phi_a))-\rk(\coker(\phi_a))\big)-\deg(L)\big(\rk\coker(\phi_a)+\rk\coker(\phi_a)\big)
\end{multline*}
and so 
\begin{equation}\label{deg}
\deg(\mathpzc{Hom}^0)-\deg(\mathpzc{Hom}^{1})\leq
\begin{cases} 
-\deg(L)\rk\coker(\phi_b) &\mbox{ if }-\deg(L)\leq \alpha\leq 0\\ 
-\deg(L)\rk\coker(\phi_a)&\mbox{ if } 0\leq\alpha\leq \deg(L).\\
 \end{cases}
\end{equation}
On the other hand we have
$$\chi(E',E)=(1-g)\big(\rk(\mathpzc{Hom}^0)-\rk(\mathpzc{Hom}^1)\big)+\deg(\mathpzc{Hom}^0)-\deg(\mathpzc{Hom}^1).$$
Combining $(\ref{deg})$ with the above equality, we get
\begin{equation*}
\chi(E',E)\leq\begin{cases} 
 (1-g)\big(\rk(\mathpzc{Hom}^0)-\rk(\mathpzc{Hom}^1)+2\rk\coker(\phi_b)\big)&\mbox{ if }-\deg(L)\leq \alpha\leq 0\\ 
(1-g)\big(\rk(\mathpzc{Hom}^0)-\rk(\mathpzc{Hom}^1)+2\rk\coker(\phi_a)\big) &\mbox{ if } 0\leq\alpha\leq \deg(L).\\
  \end{cases}
\end{equation*}
From hypothesis we have $\rk(\mathpzc{Hom}^0)\geq\rk(\mathpzc{Hom}^1)$. If $a_0$ is not generically an isomorphism then either cases of the above inequality implies $\chi(E',E)\leq(1-g)$. Otherwise, $$\chi(E',E)=\deg(\mathpzc{Hom}^0)-\deg(\mathpzc{Hom}^1)<0$$since equality happens only if $a_0$ is an isomorphism.
\end{proof}

\section{Birationality of moduli spaces}\label{bi}
Let $\alpha_c$, $\alpha_c^+$ and $\alpha_c^-$ be defined as in Section $\ref{Flip loci}$,where $\epsilon>0$ is small enough so that $\alpha_c$ is  the only critical value in the interval $(\alpha_c^-, \alpha_c^+)$. Fix a type $t=(p,q,a,b)$.
\begin{proposition}\label{codimension}
Let $\alpha_c$ be a critical value for twisted $\U(p,q)$-Higgs bundles
of type $t=(a,b,p,q)$. 
If either one of the following conditions holds:
\begin{itemize}
\item[$(1)$] $a/p-b/q>-\deg(L)$, $q\leq p$ and $0\leq\alpha_c^\pm<\frac{2pq}{pq-q^2+p+q}\big(b/q-a/p-\deg(L)\big)+\deg(L)$,
\item[$(2)$] $a/p-b/q<\deg(L)$, $p\leq q$ and $\frac{2pq}{pq-p^2+p+q}(b/q-a/p+\deg(L))-\deg(L)<\alpha_c^\pm\leq 0$.
\end{itemize}
Then the codimension of the flip loci
$\mathcal{S}_{\alpha_c^\pm}\subset \mathcal{M}^s_{\alpha_c^\pm}(t)$ is strictly positive.
\end{proposition}
\begin{proof}
{}From
  Propositions~\ref{smooth} and \ref{smooth0},
  $\mathcal{M}^s_{\alpha_c^\pm}$ is smooth of dimension $1-\chi(t,t)$.
  Hence, using that by Lemma~\ref{extension}
  $\chi(t,t)=\underset{1\leq i,j \leq m}{\sum}\chi(t_i,t_i)$, we have
 \begin{align*}
 \codim \mathcal{S}_{\alpha _c\frac{+}{}}&=\dim \mathcal{M}_{\alpha _c\frac{+}{}}^s(t)-\dim \mathcal{S}_{\alpha _c\frac{+}{}}\\
 &=1-\chi(t,t)-\dim \mathcal{S}_{\alpha _c\frac{+}{}}\\
 &=1-\sum_{i,j}\chi(t_i,t_j)-\dim \mathcal{S}_{\alpha _c\frac{+}{}},
  \end{align*}
where $t_i$, $t_i$ and $m$ occur in
$\mathrm{Gr}(E)=\bigoplus_{i=1}^mQ_i$ coming from a $\alpha_c$-Jordan-H\"{o}lder filtration of $E$. Now using the inequality $(\ref{bund})$ we get that the codimension of the strictly semistable locus is at least 
\begin{align*}
 &\min\{1-\sum_{i,j}\chi(t_i,t_j)+\sum_{i\leq
   j}\chi(t_j,t_i)+\frac{m(m-3)}{2}\}\\
     =&\min\{-\sum_{j<i}\chi(t_j,t_i)+\frac{m(m-3)+2}{2}\},
\end{align*}
where the minimum is taken over all $t_i$ and $m$. Now we show that
$Q_i$ and $Q_j$ satisfy the hypotheses of 
Lemma~\ref{bund for Euler}. 
Using Proposition~\ref{surjective}, the hypotheses $(1)$ and $(2)$ imply that $\beta$ and $\gamma$
are injective, respectively. Therefore in both cases $p_j-q_j$ and
$p_i-q_i$ have the same sign, for all $i, j$. Note that there are some
$i$ and $j$ such that the map $a_0$ of the $\mathrm{Hom}$-complex
$\mathpzc{Hom}^\bullet(Q_j,Q_i)$ is not an isomorphism, since
otherwise $\mathpzc{End}^\bullet(E)$ will be an isomorphism which is
not possible. This is because for $p\neq q$ we have
$\rk(\mathpzc{End}^0)>\rk(\mathpzc{End}^1)$ which implies that the map
$a_0$ can not be an isomorphism, and for $p = q$ it can be an
isomorphism only if $\beta$ and $\gamma$ both are isomorphisms but
this is not possible since these maps are twisted with a degree
positive line bundle.
\par Hence we have that $-\chi(t_j,t_i)> 0$ and therefore 
$$\codim \mathcal{S}_{\alpha _c\frac{+}{}}>\mbox{min}\{\frac{m(m-3)+2}{2}\}.$$
Clearly, the minimum is attained when $m=2$ giving the result. 
\end{proof}

\begin{remark}
  For $q=1$, one might have hoped to obtain a stronger result in
  Proposition~\ref{codimension}, based on (1) of
  Proposition~\ref{Prop:bound-for-dim}. The problem is that we also
  need to satisfy the hypotheses of Lemma~\ref{bund for Euler} and
  this requires injectivity of $\beta$ or $\gamma$.
\end{remark}

From Proposition~\ref{codimension} we immediately obtain the following. 
\begin{theorem}\label{birational}
Fix a type $t=(p,q,a,b)$. Let $\alpha_c$ be a critical value. Suppose that either one of the following conditions holds:
\begin{itemize}
\item[$(1)$] $a/p-b/q>-\deg(L) $, $q\leq p$ and $0\leq\alpha_c^\pm<\frac{2pq}{pq-q^2+p+q}\big(b/q-a/p-\deg(L)\big)+\deg(L)$,
\item[$(2)$] $a/p-b/q<\deg(L)$, $p\leq q$ and $\frac{2pq}{pq-p^2+p+q}(b/q-a/p+\deg(L))-\deg(L)<\alpha_c^\pm\leq 0$.
\end{itemize}
Then the moduli spaces $\mathcal{M}^s_{\alpha _c^-}(t)$ and
$\mathcal{M}^s_{\alpha _c^+}(t)$ are birationally equivalent. In
particular, if either of the conditions of
Lemma~\ref{lem:no-alpha-independent-semistability} holds then the moduli spaces $\mathcal{M}_{\alpha _c^-}(t)$ and
$\mathcal{M}_{\alpha _c^+}(t)$ are birationally equivalent.
\end{theorem}

\begin{remark}
  \label{rem:main-thm-range}
  In view of Remark~\ref{rem:alpha-0-positive}, non-emptiness of the
  intervals for $\alpha_c^\pm$ in the preceding
  theorem bounds the Toledo invariant. Thus the ranges for
  the Toledo invariant $\tau=\frac{2pq}{p+q}(a/p-b/q)$ for which the statement of the theorem is
  meaningful are:
  \begin{itemize}
  \item[(1)] $-\frac{2pq}{p+q}\deg(L)<\tau<-(q-1)\deg(L))$;
  \item[(2)] $(p-1)\deg(L)<\tau<\frac{2pq}{p+q}\deg(L)$.
  \end{itemize}
  Note that in case (1) we have $q\leq p$  and hence
  $-\frac{2pq}{p+q}\deg(L)\leq -q\deg(L)$, while in case (2) we have
  $p\leq q$ and hence $p\deg(L)\leq \frac{2pq}{p+q}\deg(L)$.
\end{remark}

Finally we have the following corollary.
\begin{theorem}\label{irreducible}
Let $L=K$ and fix a type $t=(p,q,a,b)$. Suppose that $(p+q,a+b)=1$ and
that $\tau=\frac{2pq}{p+q}(a/p-b/q)$ satisfies $\abs{\tau}\leq\min\{p,q\}(2g-2)$. Suppose that either one of the following conditions holds:
\begin{itemize}
\item[$(1)$] $a/p-b/q>-(2g-2) $, $q\leq p$ and $0\leq\alpha<\frac{2pq}{pq-q^2+p+q}\big(b/q-a/p-(2g-2)\big)+2g-2$,
\item[$(2)$] $a/p-b/q<2g-2$, $p\leq q$ and $\frac{2pq}{pq-p^2+p+q}(b/q-a/p+2g-2)-(2g-2)<\alpha\leq 0$.
\end{itemize}
Then the moduli space $\mathcal{M}_\alpha(t)$ is irreducible.  
\end{theorem}
\begin{proof}
Recall that the value of the parameter for which the non-abelian Hodge
Theorem applies is $\alpha=0$. Thus, using \cite[Theorem~6.5]{bradlow-garcia-prada-gothen:2003}, the moduli
space $\mathcal{M}_0(t)$ is irreducible and non-empty (both the
co-primality condition and the bound on the Toledo invariant are
needed for this). Hence the result follows from Theorem~\ref{birational}.
\end{proof}

\begin{remark}
  Note that unless $p=q$, the conditions on $a/b-b/q$ in the preceding theorem are
  guaranteed by the hypothesis $\abs{\tau}\leq\min\{p,q\}(2g-2)$
  (cf.~Remark~\ref{rem:main-thm-range}).
\end{remark}

\begin{remark}
  In the non-coprime case it is known from
  \cite{bradlow-garcia-prada-gothen:2003} that the
  closure of the stable locus in $\mathcal{M}_0(t)$ is connected
  (however, irreducibility is still an open question). Thus, in the
  non-coprime case, the closure of the stable locus of
  $\mathcal{M}_\alpha(t)$ is connected under the remaining
  hypotheses of the preceding theorem.
\end{remark}


\bigskip
\emph{Date:} \today

\end{document}